\documentclass[11pt, reqno]{amsart}

\usepackage{amscd,amsmath}
\usepackage{mathrsfs}
\usepackage{amsfonts}
\usepackage{amssymb}
\usepackage{enumerate}
\usepackage{setspace}
\usepackage{color}
\usepackage{esint}
\usepackage{rotating,dsfont,stackengine}
\usepackage{mathabx}
\usepackage{tabu}
\usepackage{tikz} 
\usepackage{setspace}
\usepackage[margin=1in]{geometry}

\usepackage{amssymb,extpfeil}
\usepackage{ifthen,latexsym,float,colortbl}

\usepackage{scalerel}

\usepackage{caption}
\usepackage{sidecap}

\usepackage[english]{babel}
\usepackage{graphicx}

\definecolor{luh-dark-blue}{rgb}{0.0, 0.313, 0.608}

\usepackage[colorlinks=true, citecolor=blue]{hyperref}

\newtheorem{theorem}{Theorem}[section]
\newtheorem{lemma}{Lemma}[section]

\newtheorem{remark}{Remark}[section]
\newtheorem{proposition}{Proposition}[section]

\numberwithin{equation}{section}

\DeclareMathOperator{\dv}{div}

\newcommand{\R}{\mathbb R}

\newcommand{\N}{\mathbb N}

\newcommand{\bq}{\begin{equation}}
\newcommand{\eq}{\end{equation}}

\newcommand{\lt}{\left}
\newcommand{\rt}{\right}
\newcommand{\la}{\langle}
\newcommand{\ra}{\rangle}
\newcommand{\pa}{\partial}

\begin{document}

\title[Nonlocal incompressible Navier--Stokes--Korteweg equations]{Local well-posedness and asymptotic analysis of a nonlocal incompressible Navier--Stokes--Korteweg system}

\author[Kim]{Jeongho Kim}
\address[Jeongho Kim]{\newline Department of Applied Mathematics,  \newline 
Kyung Hee University, 1732 Deogyeong-daero, Giheung-gu, Yongin-si,
Gyeonggi-do 17104, Republic of Korea}
\email{jeonghokim@khu.ac.kr}

\author[Shin]{Jaeyong Shin}
\address[Jaeyong Shin]{\newline Department of Mathematics \newline
Yonsei University, Republic of Korea}
\email{sinjaey@yonsei.ac.kr}

\date{\today}
\thanks{The work of J. Kim was supported by Samsung Science and Technology Foundation under Project Number SSTF-BA2401-01. The work of J. Shin was supported by the National Research Foundation of Korea(NRF) grant funded by the Korea government (MSIT) (RS-2024-00406821). }

\begin{abstract}
	We consider a relaxed formulation of the inhomogeneous incompressible Navier--Stokes--Korteweg system, where the classical third-order capillarity term is replaced by a nonlocal approximation. We first establish the local-in-time well-posedness of the relaxed system, under standard regularity and positivity assumptions on the initial data. The existence time is uniform with respect to both the capillarity coefficient and the relaxation parameter. We then study two asymptotic limits of the system: the nonlocal-to-local limit as the relaxation parameter tends to infinity, and the vanishing capillarity limit. In each case, we prove convergence of the solution to that of the corresponding target system. Our analysis provides a rigorous justification for the use of nonlocal relaxation models in approximating capillarity-driven incompressible fluid flows.
\end{abstract}

\maketitle

\vspace{-5ex}

\section{Introduction}

The Navier--Stokes--Korteweg (NSK) system describes the dynamics of two-phase fluids, such as liquid-vapor mixtures, with a diffuse-interface structure that incorporates capillarity effects through the Korteweg stress tensor. The NSK model traces back to the early works of van der Waals \cite{W94} and Korteweg \cite{K01}, and was later rigorously formalized by Dunn and Serrin \cite{DS85}, in which the following compressible (isentropic) NSK model was derived:   
\begin{align}\label{compressble_NSK}
\begin{aligned}
	&\pa_t \rho + \dv(\rho u) = 0,\quad t>0,\quad x\in\R^d,\\
	&\pa_t(\rho u) +\dv(\rho u \otimes u) +\nabla p=\mu\Delta u +(\lambda+\mu)\nabla\dv u+\kappa\rho\nabla\Delta \rho.
\end{aligned}
\end{align}
Here $\rho$ and $u$ denote the density and velocity of the fluid, $p$ is the pressure, $\mu,\lambda$ are viscosity coefficients and $\kappa$ is the capillarity coefficient. Since its introduction, the NSK system has been studied from various perspectives, including the local and global well-posedness \cite{AS22,DD01,H11,HL94}, asymptotic behavior \cite{TZ14}, stability of the traveling wave solutions \cite{HKKL25}, vanishing viscosity-capillarity limit \cite{CH13}, among others.

Due to its high-order nature, the capillarity term $\kappa\rho\nabla\Delta \rho$ poses difficulties for numerical discretization. In particular, it requires extended grids and may introduce numerical instabilities. To alleviate these issues, a nonlocal (or relaxed) NSK model was introduced in \cite{R05} (although the idea can be found in the work of van der Waals \cite{W94}) in which the third-order derivative term is relaxed to the nonlocal term:
\[\nabla\Delta \rho\approx \alpha^2\nabla(K_\alpha\star\rho -\rho),\]
where $K_\alpha$ is a Green function for the screened Poisson equation:
\[c = K_\alpha \star \rho,\quad\mbox{where $c$ satisfies}\quad (\alpha^2-\Delta ) c = \alpha^2\rho.\] 
The idea behind this approximation is as follows. By taking the Fourier transform to the screened Poisson equation, we have
\[(\alpha^2+|\xi|^2)\hat{c} = \alpha^2\hat{\rho},\quad\mbox{which implies}\quad \hat{K}_\alpha(\xi)=\frac{\alpha^2}{\alpha^2+|\xi|^2}.\]
Therefore, in terms of the frequency variable, $\alpha^2(K_\alpha\star\rho-\rho)$ approximates $\Delta\rho$ for large $\alpha$:
\[\alpha^2\widehat{(K_\alpha\star\rho-\rho)}=-\frac{\alpha^2|\xi|^2}{\alpha^2+|\xi|^2}\hat{\rho}\approx -|\xi|^2\hat{\rho}=\widehat{(\Delta\rho)}.\]
Therefore, it is expected that the solutions to the nonlocal NSK equations will converge to that of the local NSK equations \eqref{compressble_NSK}. This observation is rigorously justified in \cite{CH11}. We also refer to another type of parabolic relaxation for the NSK model \cite{HKMR20}.

On the other hand, there has been another stream of research that consider the incompressible counterpart of the NSK system \eqref{compressble_NSK}. Precisely, the following (inhomogeneous) incompressible Navier--Stokes--Korteweg (INSK) system has been investigated:
\begin{align} 
\begin{aligned}\label{NSK}
& \pa_t\rho+\dv(\rho u)=0, \quad t>0,\quad x\in\R^d,\\
& \rho [\pa_tu+(u\cdot\nabla u)]+\nabla p =\mu\Delta u+\kappa\rho\nabla\Delta\rho, \\
&\dv u=0.
\end{aligned}
\end{align}
As the Korteweg tensor takes into account the capillarity force that comes from the inhomogeneity of the fluid density, it is necessary to consider an inhomogeneous incompressible fluid, since the fluid density must vary for the capillarity effects to be relevant. During the last decade, many topics including the well-posedness \cite{BC17,JB23,SBGLR06,W17,Wpre,ZY10}, vanishing viscosity-capillarity limit \cite{LZZ24,WZ24,YYZ15} or blow-up criteria \cite{L20,L21} have been studied for the incompressible NSK system \eqref{NSK}.

In the present paper, we consider the nonlocal relaxation model of \eqref{NSK} given by
\begin{align}
\begin{aligned}\label{NSK-relax-0}
	&\pa_t \rho + \dv(\rho u) = 0,\quad t>0,\quad x\in\R^d,\\
	&\rho [\partial_tu + (u\cdot \nabla) u ] +\nabla p = \Delta u +\kappa\alpha^2\rho\nabla(K_\alpha\star\rho-\rho),\\
	&\dv u = 0,
\end{aligned}
\end{align}
where we normalized the viscosity coefficient $\mu=1$ for convenience. Recall that the convolution term $c:=K_\alpha\star\rho$ satisfies the following screened Poisson equation:
\[\alpha^2c-\Delta c = \alpha^2\rho.\]
Inspired by this screened Poisson equation, we define the pseudo-differential operator
\[k_\alpha(\xi) = \frac{\alpha}{\sqrt{\alpha^2+|\xi|^2}}.\]
Using the operator $k_\alpha$, we have $c = k_\alpha^2 \rho$, and the convolution term in the momentum equation of \eqref{NSK-relax-0} can be written as
\[\kappa\alpha^2\rho\nabla(K_\alpha\star\rho-\rho)=\kappa\rho\nabla(\alpha^2 c- \alpha^2\rho)= \kappa\rho\nabla\Delta c = \kappa\rho\nabla\Delta k^2_\alpha\rho.\]
Finally, observing that $\kappa\rho\nabla\Delta k^2_\alpha\rho = \nabla(\kappa\rho\Delta k^2_\alpha\rho)-\kappa\nabla\rho \Delta k^2_\alpha\rho$, the relaxed INSK system, which is the main equation of the present paper, can be written as

\begin{align}
	\begin{aligned}\label{NSK-relax}
		&\pa_t \rho + u\cdot \nabla \rho = 0,\\
		&\rho[\pa_t u + (u\cdot\nabla)u] - \Delta u +\nabla \pi = -\kappa\nabla\rho k^2_\alpha\Delta\rho,\\
		&\dv u=0,
	\end{aligned}
\end{align}
subject to initial data:
\[
(\rho, u)(0,x) = (\rho_0, u_0)(x), \quad x \in \R^d,
\]
with the modified pressure $\pi = p-\kappa\rho\Delta k^2_\alpha\rho$. Our first result is the local well-posedness of the relaxed INSK system \eqref{NSK-relax} uniformly with respect to parameters $\kappa$ or $\alpha$.

\begin{theorem}\label{thm:lwp}
Let $d=2$ or $d=3$. Assume the initial data $(\rho_0, u_0)$ satisfies $\rho_0-\bar{\rho}\in H^4(\R^d)$ for some positive constant $\bar{\rho}$, $\rho_m:=\inf{\rho_0(x)}>0$ and $\rho_M:=\sup{\rho_0(x)}<\infty$, $u_0\in H^3(\R^d)$ and $\dv u_0=0$. Moreover, assume that the capillarity coefficient $\kappa$ is bounded above, that is $0<\kappa\le \kappa_0$. Then, there exists a positive time $T_0$ which depends on $(\rho_0, u_0)$ and $\kappa_0$, but independent of $\kappa$ or $\alpha$, such that \eqref{NSK-relax} has a (unique) solution $(\rho-\bar{\rho}, u)\in C([0,T_0];H^4(\R^d))\times C([0,T_0];H^3(\R^d))$.
\end{theorem}

\begin{remark}
The function space $H^4(\R^d)\times H^3(\R^d)$ for $(\rho(t)-\bar{\rho},u(t))$ in Theorem \ref{thm:lwp} can be replaced by $H^{s+1}(\R^d)\times H^s(\R^d)$ for any integer $s>\frac{d}{2}+1$. However, for simplicity, we set $s=3$ throughout this paper, including in Theorem \ref{thm:lwp} and the subsequent results.
\end{remark}

After obtaining the local well-posedness of the nonlocal INSK system, we consider two asymptotic limits. We first study the so-called ``nonlocal-to-local" limit, where we take a limit $\alpha\to+\infty$. As we already mentioned, it is expected that the solutions to the nonlocal INSK system \eqref{NSK-relax} will converge to the solution of the local (or standard) INSK system \eqref{NSK}. 

\begin{theorem}\label{thm:nonlocal_to_local}
	Let $(\rho^\alpha,u^\alpha)$ be the local solution to \eqref{NSK-relax} constructed in Theorem \ref{thm:lwp} and let $(\rho-\bar{\rho},u)\in C([0,T_0];H^4(\R^d))\times C([0,T_0];H^3(\R^d))$ be the local solution to \eqref{NSK} subject to the same initial data $(\rho_0,u_0)$. Then, for each $l\in[0,2]$ there exists positive constant $C_l$ such that
	\[\sup_{0\le t\le T_0}\left(\|\rho^\alpha(t)-\rho(t)\|_{H^{2+l}(\R^d)}^2+\|u^\alpha(t)-u(t)\|_{H^{1+l}(\R^d)}^2\right)\le \frac{C_l(T_0)}{\alpha^{4-2l}}.\]
\end{theorem}

\begin{remark}
	The local existence of the solution $(\rho,u)$ of \eqref{NSK} with desired regularity is already proved in \cite{YYZ15}. For convenience, we assume that the lifespan of the solution $(\rho,u)$ for \eqref{NSK} is also $T_0$ in Theorem \ref{thm:nonlocal_to_local}. In general, we may choose $T_0$ as a minimum of the lifespans of the solutions $(\rho^\alpha,u^\alpha)$ and $(\rho,u)$. The same assumption is adopted in Theorem \ref{thm:vanishing_capillarity} below.
\end{remark}

The second asymptotic limit that we consider is vanishing capillarity limit. In this case, we fix $\alpha$ and then take $\kappa\to0$. After taking $\kappa\to0$ in \eqref{NSK-relax}, it is expected that the solution to \eqref{NSK-relax} converges to the solution of the (inhomogeneous) incompressible Navier--Stokes equations:
\begin{align}
\begin{aligned}\label{NS}
	&\pa_t\rho+\dv(\rho u) = 0,\quad t>0,\quad x\in\R^d,\\
	&\rho[\pa_t u+(u\cdot\nabla)u]-\Delta u +\nabla\pi = 0,\\
	&\dv u =0.
\end{aligned}
\end{align}
Our second asymptotic limit is to justify this vanishing capillarity limit.
\begin{theorem}\label{thm:vanishing_capillarity}
	Let $(\rho^\kappa,u^\kappa)$ be the local solution to \eqref{NSK-relax} constructed in Theorem \ref{thm:lwp} with the capillarity coefficient $\kappa$ and let $(\rho-\bar{\rho},u)\in C([0,T_0];H^3(\R^3))\times C([0,T_0];H^3(\R^3))$ be the local solution to \eqref{NS} subject to the same initial data $(\rho_0,u_0)$. Then, there exists positive constant $C$ such that
	\[\sup_{0\le t\le T_0}\left(\|\rho^\kappa(t)-\rho(t)\|_{H^3(\R^d)}^2+\|u^\kappa(t)-u(t)\|_{H^3(\R^d)}^2\right)\le C(T_0)\kappa^2.\]
\end{theorem}

The remaining part of the paper is organized as follows. In Section \ref{sec:2}, we provide several preliminary tools that will be used in the later analysis. Section \ref{sec:3} presents the local well-posedness of the nonlocal INSK system \eqref{NSK-relax} based on the iterative method. Then, in Section \ref{sec:4} and Section \ref{sec:5}, we deal with the convergence of the nonlocal INSK system \eqref{NSK-relax} towards the local INSK system \eqref{NSK} and the incompressible Navier--Stokes system \eqref{NS} as $\alpha\to+\infty$ and as $\kappa\to0$ respectively.

\subsection*{Notation}
In the following, we omit the domain of integration when we take the integral over $\R^d$, that is, $\int:=\int_{\R^d}$. We also omit the domain in the function spaces and use the abbreviation $L^p:=L^p(\R^d)$ and $H^s:=H^s(\R^d)$. We denote the standard $L^2$-inner product between $f,g\in L^2$ as $\la f,g\ra$. A generic constant $C$ may differ from line to line, unless we specify the constant.

\section{Preliminaries}\label{sec:2}

In this section, we collect several preliminary lemmas that will be used in the subsequent analysis. We begin by stating some properties of the pseudo-differential operator $k_\alpha$.

\begin{lemma}\label{lem:est-k_alpha}
	For any $f\in H^s$, we have
	\begin{equation*}
		\|k_\alpha f\|_{H^s}^2 \le \|f\|_{H^s}^2,\quad \mbox{and}\quad  \|\nabla k^2_\alpha f\|_{H^s}^2\le \alpha^2\|f\|_{H^s}^2. 
	\end{equation*}
	Moreover, for any $f\in H^{s+l}$, $0\leq l \leq 2$, it holds that
	\[\|(k^2_\alpha-I)f\|_{H^s}^2\le \frac{1}{\alpha^{2l}}\|f\|_{H^{s+l}}^2.\]
	Finally, $k_\alpha$ is a self-adjoint operator on $L^2$.
\end{lemma}

\begin{proof}
	Using the definition of $k_\alpha$, we get
	\[\|k_\alpha f\|_{H^s}^2 = \int \left(\frac{\alpha^2}{\alpha^2+|\xi|^2}\right)(1+|\xi|^2)^s|\hat{f}(\xi)|^2\,d\xi\le \int(1+|\xi|^2)^s|\hat{f}(\xi)|^2\,d\xi=\|f\|_{H^s}^2.\]
	Similarly, we obtain
	\[\|\nabla k_\alpha^2 f\|_{H^s}^2 = \int \left(\frac{\alpha^2}{\alpha^2+|\xi|^2}\right)^2|\xi|^2(1+|\xi|^2)^s|\hat{f}(\xi)|^2\,d\xi\le \alpha^2\int (1+|\xi|^2)^s|\hat{f}(\xi)|^2\,d\xi\le\alpha^2\|f\|_{H^s}^2.\]
	Next, we again use the definition of $k_\alpha$ to obtain
	\begin{align*}
		\|(k_\alpha^2-I)f\|_{H^s}^2&=\int \left(\frac{|\xi|^2}{\alpha^2+|\xi|^2}\right)^2(1+|\xi|^2)^s|\hat{f}(\xi)|^2\,d\xi \\
		&\le \frac{1}{\alpha^{2l}}\int |\xi|^{2l}(1+|\xi|^2)^{s}|\hat{f}(\xi)|^2\,d\xi=\frac{1}{\alpha^{2l}}\|f\|_{H^{s+l}}^2.
	\end{align*}
	Finally, it is straightforward to show that $k_\alpha$ is a self-adjoint operator on $L^2$, since
	\[\la k_\alpha f,g\ra=\int \frac{\alpha}{\sqrt{\alpha^2+|\xi|^2}}\hat{f}(\xi)\hat{g}(\xi)\,d\xi=\la f,k_\alpha g\ra.\]
\end{proof}

Next, we recall the Gagliardo--Nirenberg--Sobolev inequality, as well as commutator and product estimates in Sobolev spaces.

\begin{lemma}\label{lem:sobolev}
	Let $p,q,r\in [1,+\infty]$ and $i,j,m$ be non-negative integers such that $i,j\le m$. Then, there exists a constant $\theta\in[0,1]$ and $C>0$ such that 
	\[\|D^ju\|_{L^p} \le C\|D^m u\|_{L^r}^\theta \|D^iu\|_{L^q}^{1-\theta},\]
	where
	\[\frac{1}{p}=\frac{j}{d}+\theta\left(\frac{1}{r}-\frac{m}{d}\right)+(1-\theta)\left(\frac{1}{q}-\frac{i}{d}\right).\]
\end{lemma}

\begin{lemma}\label{lem:commutator}
	For any multi-index $\gamma=(\gamma_1,\gamma_2,\gamma_3)$ with $|\gamma|:=\gamma_1+\gamma_2+\gamma_3$, let $D^\gamma$ be the differential operator $D^\gamma = \pa_1^{\gamma_1}\pa_2^{\gamma_2}\pa_3^{\gamma_3}$. Then, for any $f\in H^s$, $\nabla f\in L^\infty$ and $g\in H^{s-1}\cap L^\infty$, $s\in\N$, we have 
	\[\sum_{|\gamma|\le s} \|D^\gamma (fg)-fD^\gamma g\|_{L^2}\le C (\|f\|_{H^s}\|g\|_{L^\infty}+\|\nabla f\|_{L^\infty}\|g\|_{H^{s-1}}).\]
	Moreover, for $f,g\in H^s\cap L^\infty$, we have
	\[\sum_{|\gamma|\le s} \|D^\gamma(fg)\|_{L^2}\le C(\|f\|_{H^s}\|g\|_{L^\infty}+\|f\|_{L^\infty}\|g\|_{H^s}).\]
\end{lemma}

\section{Local well-posedness of the relaxed incompressible NSK system} \label{sec:3}

In this section, we present the proof of Theorem \ref{thm:lwp}, the local well-posedness of the relaxed INSK system \eqref{NSK-relax}. The proof is divided into two main steps. We first show that \eqref{NSK-relax} has a unique local solution, whose existence time may depend on $\alpha$, by considering a sequence of approximate solutions (Section \ref{sec:3-1}--Section \ref{sec:3-2}). Then, we show that the existence time is indeed bounded from below uniformly in $\alpha$ (Section \ref{sec:3-3}) by careful energy estimate.

To this end, we consider a sequence of approximate solutions $(\rho^n,u^n,\eta^n)$ defined as the solution of the following linearized system

\begin{subequations}\label{INSK-relax-approx}
	\begin{align}
		&\pa_t\rho^{n+1}+u^{n+1}\cdot\nabla\rho^{n+1}=0, \label{INSK-relax-approx-a}\\
		&\rho^n[\pa_t u^{n+1}+(u^n\cdot\nabla)u^{n+1}] -\Delta u^{n+1}+\nabla \pi^{n+1}=-\kappa\eta^n \dv k_\alpha^2\eta^{n+1}, \label{INSK-relax-approx-b}\\
		&\pa_t \eta_j^{n+1}+u^n\cdot \nabla\eta_j^{n+1}+\pa_j u^{n+1}\cdot \eta^n=0, \label{INSK-relax-approx-c}\\
		&\dv u^{n+1}=0\label{INSK-relax-approx-d}
	\end{align}
\end{subequations}

subject to the initial data
\[\rho^n(0,x) = \rho_0(x),\quad u^n(0,x)=u_0(x),\quad \eta^n(0,x)=\eta_0(x)=\nabla\rho_0(x).\]

\subsection{Uniform boundedness of the approximate solutions}
\label{sec:3-1}

We first show that the sequence $(\rho^n-\bar{\rho},u^n,\eta^n)$ is bounded in $H^3$. First of all, it follows from \eqref{INSK-relax-approx-a} that the density is constant along the Lagrangian frame:
\[\rho^{n+1}(t,x)=\rho_0(X(0)),\quad\mbox{where}\quad \frac{dX}{ds} = u^{n+1}(s,X(s)),\quad X(t)=x,\]
which immediately implies the uniform bound $\rho_m\le \rho^{n+1}(t,x)\le \rho_M$ for all $t>0$ and $x\in\R^3$. We start with the $H^3$-estimates of $\tilde\rho^n = \rho^n-\bar{\rho}$ and $\eta^n$.

\begin{lemma}\label{lem:tilderho_H3}
	Under the assumption of Theorem \ref{thm:lwp}, we have
	\[\frac{d}{dt}\|\tilde{\rho}^{n+1}\|_{H^3}^2\le C\|u^{n+1}\|_{H^3}\|\tilde{\rho}^{n+1}\|_{H^3}^2.\]
\end{lemma}

\begin{proof}
	First, it follows from \eqref{INSK-relax-approx-a} that $\tilde{\rho}^{n+1}$ also satisfies
	\begin{equation}\label{tilde-rho}
		\pa_t\tilde{\rho}^{n+1}+u^{n+1}\cdot\nabla\tilde{\rho}^{n+1}=0.
	\end{equation}
	We take $D^\gamma$ with $|\gamma|\le 3$ to \eqref{tilde-rho}, take an $L^2$-inner product with $D^\gamma\tilde{\rho}^{n+1}$ and Lemma \ref{lem:sobolev}, Lemma \ref{lem:commutator} to obtain
	\begin{align*}
		\frac{1}{2}\frac{d}{dt}\|D^\gamma\tilde{\rho}^{n+1}\|_{L^2}^2&=-\la D^\gamma(u^{n+1}\cdot\nabla\tilde{\rho}^{n+1})-u^{n+1}\cdot\nabla D^\gamma\tilde{\rho}^{n+1},D^\gamma\tilde{\rho}^{n+1}\ra\\
		&\le C(\|u^{n+1}\|_{H^3}\|\nabla\tilde{\rho}^{n+1}\|_{L^\infty}+\|\nabla u^{n+1}\|_{L^\infty}\|\nabla\tilde{\rho}^{n+1}\|_{H^2})\|\tilde{\rho}^{n+1}\|_{H^3}\\
		&\le C\|u^{n+1}\|_{H^3}\|\tilde{\rho}^{n+1}\|_{H^3}^2,
	\end{align*}
	where we used the incompressibility of $u^{n+1}$. Summing over all multi-indices $|\gamma|\le 3$, we obtain the desired estimate.
\end{proof}

\begin{lemma}\label{lem:eta_H3}
	Under the assumption of Theorem \ref{thm:lwp}, there exists a sufficiently small positive constant $\delta_0$ such that
	\[\frac{d}{dt}\|\eta^{n+1}\|_{H^3}^2\le C \|u^n\|_{H^3} \|\eta^{n+1}\|_{H^3}^2+C\|\eta^n\|_{H^3}^2\|\eta^{n+1}\|_{H^3}^2+\delta_0\|\nabla u^{n+1}\|^2_{H^3}.\]
\end{lemma}

\begin{proof}
	Applying $D^\gamma$ with $|\gamma|\le 3$ to \eqref{INSK-relax-approx-c}, we take $L^2$-inner product with $D^\gamma \eta^{n+1}_j$ to obtain
	\begin{align*}
		\frac{1}{2}\frac{d}{dt}\|D^\gamma\eta_j^{n+1}\|_{L^2}^2&=-\la D^\gamma(u^n\cdot\nabla\eta_j^{n+1}),D^\gamma\eta_j^{n+1}\ra - \la D^\gamma(\pa_j u^{n+1}\cdot\eta^n),D^\gamma \eta_j^{n+1}\ra\\
		&=-\la D^\gamma(u^n\cdot\nabla \eta_j^{n+1})-u^n\cdot \nabla D^\gamma \eta_j^{n+1},D^\gamma\eta_j^{n+1}\ra- \la D^\gamma(\pa_j u^{n+1}\cdot\eta^n),D^\gamma \eta_j^{n+1}\ra\\
		&\le C(\|u^n\|_{H^3}\|\nabla\eta^{n+1}\|_{L^\infty}+\|\nabla u^n\|_{L^\infty}\|\nabla\eta^{n+1}\|_{H^2})\|\eta^{n+1}\|_{H^3}\\
		&\quad + C(\|\nabla u^{n+1}\|_{H^3}\|\eta^n\|_{L^\infty}+\|\nabla u^{n+1}\|_{L^\infty}\|\eta^n\|_{H^3})\|\eta^{n+1}\|_{H^3}\\
		&\le C\|u^n\|_{H^3}\|\eta^{n+1}\|_{H^3}^2 + C\|\nabla u^{n+1}\|_{H^3}\|\eta^n\|_{H^3}\|\eta^{n+1}\|_{H^3}\\
		&\le C \|u^n\|_{H^3} \|\eta^{n+1}\|_{H^3}^2+C(\delta)\|\eta^n\|_{H^3}^2\|\eta^{n+1}\|_{H^3}^2+\delta\|\nabla u^{n+1}\|_{H^3}^2,
	\end{align*}
	for small enough constant $\delta$. Thus, gathering the estimate for all $|\gamma|\le 3$, we can find a sufficiently small constant $\delta_0$ that the desired estimate holds.
\end{proof}

We now derive the following $H^3$-estimate for $u^{n+1}$.

\begin{lemma}\label{lem:u_H3}
	Under the assumption of Theorem \ref{thm:lwp}, we have
	\begin{align*}
		&\|u^{n+1}(t)\|_{H^3}^2+\kappa\|k_\alpha \eta^{n+1}(t)\|_{H^3}^2 + \int_0^t \|\nabla u^{n+1}\|_{H^3}^2\,ds\\
		&\le C(\|u_0\|_{H^3}^2+\kappa\|k_\alpha\eta_0\|_{H^3}^2)\\
		&\quad +C(1+\alpha)\int_0^t (1+\|\tilde{\rho}^n\|_{H^3}^2+\|u^n\|_{H^3}^2+\|\eta^n\|_{H^3}^2)(\|u^{n+1}\|_{H^3}^2+\|\eta^{n+1}\|_{H^3}^2)\,ds+\int_0^t \|\pa_t u^{n+1}\|_{H^2}^2\,ds.
	\end{align*}
\end{lemma}

\begin{proof}
	
Again, we apply $D^\gamma$ with $|\gamma|\le 3$ to \eqref{INSK-relax-approx-b} and \eqref{INSK-relax-approx-c} and then take an $L^2$-inner product with $D^\gamma u^{n+1}$ and $\kappa D^\gamma k_\alpha^2\eta_j^{n+1}$ respectively to obtain 
\begin{align}
	\begin{aligned}\label{est:H3-u}
		&\frac{1}{2}\frac{d}{dt}\lt(\int\rho^{n}|D^\gamma u^{n+1}|^2\,dx+\kappa\|k_\alpha D^\gamma\eta^{n+1}\|_{L^2}^2\rt)+\|\nabla D^\gamma u^{n+1}\|_{L^2}^2\\
		&=-\frac{1}{2}\int \dv(\rho^n u^n)|D^\gamma u^{n+1}|^2\,dx -\la D^\gamma(\rho^n\pa_t u^{n+1})-\rho^nD^\gamma\pa_t u^{n+1},D^\gamma u^{n+1}\ra\\
		&\quad -\lt\la D^\gamma\lt((\rho^{n}u^n\cdot\nabla)u^{n+1}\rt),D^\gamma u^{n+1}\rt\ra -\lt\la \kappa D^\gamma(\eta^n\pa_jk_\alpha^2\eta^{n+1}_j),D^\gamma u^{n+1}\rt\ra \\
		&\quad -\lt\la D^\gamma(u^n\cdot\nabla \eta^{n+1}_j),\kappa D^\gamma k_\alpha^2\eta^{n+1}_j\rt\ra -\lt\la D^\gamma(\pa_j u^{n+1}\cdot\eta^n),\kappa D^\gamma k_\alpha^2\eta^{n+1}_j\rt\ra =:\sum_{i=1}^6 I_{1i}.
	\end{aligned}
\end{align}
	
Then, we use Lemma \ref{lem:sobolev} and Lemma \ref{lem:commutator} to bound $I_{11}$, $I_{12}$, and $I_{13}$ as
	
\begin{align*}
	I_{11}+I_{13}&= -\lt\la D^\gamma\lt((\rho^{n}u^n\cdot\nabla)u^{n+1}\rt)-(\rho^n u^n\cdot\nabla)D^\gamma u^{n+1},D^\gamma u^{n+1}\rt\ra\\
	&\le C(\|\rho^nu^n\|_{H^3}\|\nabla u^{n+1}\|_{L^\infty}+\|\nabla(\rho^n u^n)\|_{L^\infty}\|\nabla u^{n+1}\|_{H^2})\|u^{n+1}\|_{H^3}\\
	&\le C(1+\|\tilde{\rho}^n\|_{H^3})\|u^n\|_{H^3}\|u^{n+1}\|_{H^3}^2, \\
	I_{12}&\le C(\|\tilde{\rho}^n\|_{H^3}\|\pa_t u^{n+1}\|_{L^\infty}+\|\nabla\tilde{\rho}^n\|_{L^\infty}\|\pa_tu^{n+1}\|_{H^2})\|u^{n+1}\|_{H^3}\\
	&\le C\|\tilde{\rho}^n\|_{H^3}\|\pa_t u^{n+1}\|_{H^2}\|u^{n+1}\|_{H^3}.
\end{align*}
Next, since $\kappa$ is bounded, we estimate $I_{15}$ as
\begin{align*}
	I_{15} &= -\la D^\gamma(u^n\cdot\nabla\eta_j^{n+1})-u^n\cdot\nabla D^\gamma\eta_j^{n+1},\kappa D^\gamma k_\alpha^2\eta_j^{n+1}\ra-\la u^n\cdot\nabla D^\gamma \eta_j^{n+1},\kappa D^\gamma k_\alpha^2\eta_j^{n+1}\ra\\
	&\le C(\|u^n\|_{H^3}\|\nabla\eta^{n+1}\|_{L^\infty}+\|\nabla u^n\|_{L^\infty}\|\nabla\eta^{n+1}\|_{H^2})\|D^\gamma k^2_\alpha\eta^{n+1}\|_{L^2}\\
	&\quad +\kappa\la u^n D^\gamma \eta_j^{n+1},\nabla D^\gamma k_\alpha^2\eta_j^{n+1}\ra\\
	&\le C\|u^n\|_{H^3}\|\eta^{n+1}\|_{H^3}\|k^2_\alpha \eta^{n+1}\|_{H^3}+ C\|u^n\|_{L^\infty}\|D^\gamma \eta^{n+1}\|_{L^2}\|\nabla D^\gamma k^2_\alpha \eta^{n+1}\|_{L^2}.
\end{align*}
However, we use Lemma \ref{lem:est-k_alpha} to obtain
\[\|k^2_\alpha\eta^{n+1}\|_{H^3}\le \|\eta^{n+1}\|_{H^3},\quad \mbox{and}\quad \|\nabla D^\gamma k^2_\alpha \eta^{n+1}\|_{L^2}\le \alpha \|D^\gamma \eta^{n+1}\|_{L^2}.\]
This yields
\[I_{15}\le C\|u^n\|_{H^3}\|\eta^{n+1}\|_{H^3}^2+ C\alpha\|u^n\|_{H^3}\|\eta^{n+1}\|_{H^3}^2=C(1+\alpha)\|u^n\|_{H^3}\|\eta^{n+1}\|_{H^3}^2.\]
Finally, we combine $I_{14}$ and $I_{16}$ to estimate them as

\begin{align*}
	I_{14}+I_{16} &= -\kappa\la D^\gamma(\eta^n\pa_j k^2_\alpha \eta_j^{n+1})-\eta^nD^\gamma \pa_jk^2_\alpha \eta_j^{n+1},D^\gamma u^{n+1}\ra -\kappa \la \eta^n D^\gamma \pa_j k^2_\alpha \eta^{n+1}_j,D^\gamma u^{n+1}\ra\\
	&\quad -\kappa\la D^\gamma(\pa_j u^{n+1}\cdot\eta^n)-\eta^n\cdot D^\gamma \pa_j u^{n+1},D^\gamma k^2_\alpha \eta^{n+1}_j\ra-\kappa\la \eta^n\cdot D^\gamma \pa_j u^{n+1},D^\gamma k_\alpha^2 \eta^{n+1}_j\ra\\
	&\le C(\|\eta^n\|_{H^3}\|\pa_j k_\alpha^2\eta_j^{n+1}\|_{L^\infty}+\|\nabla\eta^n\|_{L^\infty} \|\pa_j k^2_\alpha \eta_j^{n+1}\|_{H^2})\|D^\gamma u^{n+1}\|_{L^2}\\
	&\quad + C(\|\eta^n\|_{H^3}\|\pa_j u^{n+1}\|_{L^\infty}+\|\nabla\eta^n\|_{L^\infty}\|\pa_j u^{n+1}\|_{H^2})\|D^\gamma k^2_\alpha \eta^{n+1}_j\|_{L^2} +\kappa\la \pa_j \eta^n D^\gamma k^2_\alpha \eta^{n+1}_j,D^\gamma u^{n+1}\ra\\
	&\le C\|\eta^n\|_{H^3}\|\eta^{n+1}\|_{H^3}\|u^{n+1}\|_{H^3},
\end{align*}
where we used Lemma \ref{lem:est-k_alpha}. Thus, gathering all the estimates for $I_{1i}$ for $i=1,2,\ldots, 6$ with \eqref{est:H3-u}, we obtain
\begin{align*}
	&\frac{1}{2}\frac{d}{dt}\left(\int \rho^n|D^\gamma u^{n+1}|^2\,dx + \kappa\|k_\alpha D^\gamma \eta^{n+1}\|_{L^2}^2\right)+\|\nabla D^\gamma u^{n+1}\|_{L^2}^2\\
	&\le C(1+\|\tilde{\rho}^n\|_{H^3})\|u^n\|_{H^3}\|u^{n+1}\|_{H^3}^2+C\|\tilde{\rho}^n\|_{H^3}\|\pa_t u^{n+1}\|_{H^2}\|u^{n+1}\|_{H^3}\\
	&\quad+C(1+\alpha)\|u^n\|_{H^3}\|\eta^{n+1}\|_{H^3}^2+C\|\eta^n\|_{H^3}\|\eta^{n+1}\|_{H^3}\|u^{n+1}\|_{H^3}\\
	&\le C(1+\alpha)(1+\|\tilde{\rho}^n\|_{H^3}^2+\|u^n\|_{H^3}^2+\|\eta^n\|_{H^3}^2)(\|u^{n+1}\|_{H^3}^2+\|\eta^{n+1}\|_{H^3}^2)+\|\pa_t u^{n+1}\|_{H^2}^2.
\end{align*}
	
We combine above estimate over all $|\gamma|\le 3$, take time integration, and use the fact that $\rho_m\le \rho^n\le \rho_M$ to get
	
\begin{align*}
	&\|u^{n+1}(t)\|_{H^3}^2+\kappa\|k_\alpha \eta^{n+1}(t)\|_{H^3}^2 + \int_0^t \|\nabla u^{n+1}(s)\|_{H^3}^2\\
	&\le C(\|u_0\|_{H^3}^2+\kappa\|k_\alpha\eta_0\|_{H^3}^2)\\
	&\quad +C(1+\alpha)\int_0^t (1+\|\tilde{\rho}^n\|_{H^3}^2+\|u^n\|_{H^3}^2+\|\eta^n\|_{H^3}^2)(\|u^{n+1}\|_{H^3}^2+\|\eta^{n+1}\|_{H^3}^2)\,ds+\int_0^t \|\pa_t u^{n+1}\|_{H^2}^2\,ds.
\end{align*}
	
\end{proof}

Next, we need to control the term $\int_0^t\|\pa_tu^{n+1}\|_{H^2}^2\,ds$. The following two lemmas provide the desired estimate on it.

\begin{lemma}\label{lem:patu}
	Under the same assumption of Theorem \ref{thm:lwp}, there exists a positive constant $C_1$ that satisfying
	\[\int_0^t \|\pa_t u^{n+1}\|_{L^2}^2\,ds +\|\nabla u^{n+1}(t)\|_{L^2}^2 \le C_1\|\nabla u_0\|_{L^2}^2 + C\int_0^t (\|u^n\|_{H^3}^2+\|\eta^n\|_{H^3}^2)(\|u^{n+1}\|_{H^3}^2+\|\eta^{n+1}\|_{H^3}^2)\,ds.\]
\end{lemma}

\begin{proof}
	We multiply \eqref{INSK-relax-approx-b} by $\pa_t u^{n+1}$ and use the boundedness of $k_\alpha$ to obtain
	\[
	\begin{aligned}
		\int \rho^{n}|\pa_t u^{n+1}|^2\,dx+ \frac{1}{2}\frac{d}{dt}\|\nabla u^{n+1}\|_{L^2}^2 &= -\lt\la (\rho^{n}u^n\cdot\nabla)u^{n+1},\pa_t u^{n+1}\rt\ra -\lt\la \kappa\eta^n\dv k_\alpha^2\eta^{n+1},\pa_t u^{n+1}\rt\ra \\
		&\leq \lt( \|\rho^{n}\|_{L^\infty}\|u^n\|_{L^\infty}\|\nabla u^{n+1}\|_{L^2}+\kappa\|\eta^n\|_{L^\infty}\|\dv k_\alpha^2\eta^{n+1}\|_{L^2}\rt)\|\pa_t u^{n+1}\|_{L^2} \\
		&\leq \frac{\rho_m}{2}\|\pa_t u^{n+1}\|_{L^2}^2+ C\lt(\|u^n\|_{H^3}^2+\|\eta^n\|_{H^3}^2\rt)\lt(\|u^{n+1}\|_{H^3}^2+\|\eta^{n+1}\|_{H^3}^2\rt).
	\end{aligned}
	\]
	Since $\rho_m\le\rho^n\le \rho_M$, we obtain the desired estimate after taking time integration.
\end{proof}

\begin{lemma}\label{lem:patu_H2}
	Under the same assumption of Theorem \ref{thm:lwp}, there exists positive constant $C_2$ such that we have
	\begin{align*}
		&\|\nabla u^{n+1}(t)\|_{H^2}^2 + \int_0^t \|\pa_t u^{n+1}\|_{H^2}^2\,ds \\
		&\le C_2\|\nabla u_0\|_{H^2}^2+C\int_0^t(1+\|\tilde{\rho}^n\|_{H^3}^2)(\|u^n\|_{H^3}^2+\|\eta^n\|_{H^3}^2)(\|u^{n+1}\|_{H^3}^2+\|\eta^{n+1}\|_{H^3}^2)+C_2\|\tilde{\rho}^n\|_{H^3}^4\|\pa_tu^{n+1}\|_{L^2}^2\,ds.
	\end{align*}
\end{lemma}

\begin{proof}
	Applying $D^\gamma$ with $|\gamma|\le 2$ to \eqref{INSK-relax-approx-b} and taking $L^2$-inner product by $D^\gamma \pa_t u^{n+1}$, we have
	\[
	\begin{aligned}
		&\int \rho^{n}|D^\gamma \pa_t u^{n+1}|^2\,dx+\frac{1}{2}\frac{d}{dt}\|\nabla D^\gamma u^{n+1}\|_{L^2}^2 \\
		&= -\lt\la D^\gamma\lt((\rho^{n}u^n\cdot\nabla)u^{n+1}\rt),D^\gamma\pa_t u^{n+1}\rt\ra -\lt\la D^\gamma(\kappa\eta^n\dv k_{\alpha}^2\eta^{n+1}),D^\gamma\pa_t u^{n+1}\rt\ra \\
		&\quad -\lt\la D^\gamma(\rho^{n}\pa_t u^{n+1})-\rho^{n}D^\gamma\pa_t u^{n+1},D^\gamma\pa_t u^{n+1}\rt\ra  =:I_{21}+I_{22}+I_{23}.
	\end{aligned}
	\]
	Using the boundedness of $k_\alpha$, we easily compute
	\[
	\begin{aligned}
		I_{21}+I_{22} &\leq C\lt(1+\|\tilde{\rho}^{n}\|_{H^3}\rt)\lt(\|u^n\|_{H^3}+\|\eta^n\|_{H^3}\rt)\lt(\|u^{n+1}\|_{H^3}+\|\eta^{n+1}\|_{H^3}\rt)\|\pa_t u^{n+1}\|_{H^2} \\
		&\leq \frac{\rho_m}{4}\|\pa_t u^{n+1}\|_{H^2}^2+C\lt(1+\|\tilde{\rho}^{n}\|_{H^3}^2\rt)\lt(\|u^n\|_{H^3}^2+\|\eta^n\|_{H^3}^2\rt)\lt(\|u^{n+1}\|_{H^3}^2+\|\eta^{n+1}\|_{H^3}^2\rt).
	\end{aligned}
	\]
	On the other hand, we use the Leibniz rule and Lemma \ref{lem:sobolev} to estimate $I_{23}$ as 
	\[
	\begin{aligned}
		I_{23} &\leq C\|\nabla \rho^n\|_{L^\infty} \|\pa_t u^{n+1}\|_{L^2}\|\pa_t u^{n+1}\|_{H^1} \\
		&\quad+ C\lt( \|\nabla^2\rho^{n}\|_{L^4}\|\pa_t u^{n+1}\|_{L^4}+\|\nabla \rho^{n}\|_{L^\infty}\|\nabla\pa_t u^{n+1}\|_{L^2}\rt)\|\pa_t u^{n+1}\|_{H^2}\\
		&\leq C\|\tilde{\rho}^n\|_{H^3}\|\pa_t u^{n+1}\|_{H^1}\|\pa_t u^{n+1}\|_{H^2}\leq C\|\tilde{\rho}^{n}\|_{H^3}\|\pa_t u^{n+1}\|_{L^2}^{1/2}\|\pa_t u^{n+1}\|_{H^2}^{3/2} \\
		&\leq \frac{\rho_m}{4} \|\pa_t u^{n+1}\|_{H^2}^2+C\|\tilde{\rho}^{n}\|_{H^3}^4\|\pa_t u^{n+1}\|_{L^2}^2.
	\end{aligned}
	\]
	Again, using the boundedness of $\rho^n$, we obtain the desired estimate.
\end{proof}

Combining the estimates from the previous lemma, we obtain the following uniform boundedness of the $H^3$-norms of $(\tilde{\rho}^n,u^n,\eta^n)$.

\begin{proposition}\label{prop:uniform_bound}
	Let $M_1$ and $M_2$ be two positive constants that satisfying
	\[M_2>\|\tilde{\rho}_0\|_{H^3}^2,\quad M_1>C(\|u_0\|_{H^3}^2+\kappa\|k_\alpha\eta_0\|_{H^3}^2 + \|\eta_0\|_{H^3}^2)+(C_2+C_1C_2M_2^2)\|\nabla u_0\|_{H^2}^2,\]
	where the constants $C_1$ and $C_2$ are the same constants in Lemma \ref{lem:patu} and Lemma \ref{lem:patu_H2}. Then, there exists $T_*$, which depends on $M_1,M_2,\alpha$ but independent of $n$ such that
	\begin{equation}\label{uniform_bound}
	\begin{split}
		\sup_{n\in \N}\sup_{0\le t\le T_*}\lt(\|u^n(t)\|_{H^3}^2+\|\eta^n(t)\|_{H^3}^2+\int^t_0\|\nabla u^n\|_{H^3}^2+\|\pa_t u^n\|_{H^2}^2\,ds\rt)\le M_1,\\
		\sup_{n\in\N}\sup_{0\le t\le T_*}\|\tilde{\rho}^n(t)\|_{H^3}^2\le M_2.
	\end{split}
	\end{equation}
\end{proposition}

\begin{proof}
	For each $n\in\N$, suppose
	\[\|u^n(t)\|_{H^3}^2+\|\eta^n(t)\|_{H^3}^2+\int^t_0\|\nabla u^n\|_{H^3}^2+\|\pa_t u^n\|_{H^2}^2\,ds\le M_1,\quad \|\tilde{\rho}^n(t)\|_{H^3}^2\le M_2,\quad t\in [0,T_*].\]
	Then, combining Lemma \ref{lem:eta_H3}, Lemma \ref{lem:u_H3}, and Lemma \ref{lem:patu_H2}, we have
	\begin{align*}
		&\|u^{n+1}(t)\|_{H^3}^2+\kappa\|k_\alpha\eta^{n+1}(t)\|_{H^3}^2 + \|\eta^{n+1}(t)\|_{H^3}^2 + \int_0^t \|\nabla u^{n+1}\|_{H^3}^2 + \|\pa_t u^{n+1}\|_{H^2}^2\,ds\\
		&\le C(\|u_0\|_{H^3}^2+\kappa\|k_\alpha\eta_0\|_{H^3}^2+\|\eta_0\|_{H^3}^2) +C_2\|\nabla u_0\|_{H^2}^2+ C(\alpha,M_1,M_2)\int_0^t(\|u^{n+1}\|_{H^3}^2+\|\eta^{n+1}\|_{H^3}^2)\,ds\\
		&\quad +C_2M_2^2\int_0^t\|\pa_t u^{n+1}\|_{L^2}^2\,ds,
	\end{align*}
	where $C(\alpha,M_1,M_2)$ is a constant that depends on $\alpha$, $M_1$, and $M_2$. Then, we use Lemma \ref{lem:patu} to obtain
	\begin{align*}
		&\|u^{n+1}(t)\|_{H^3}^2+\kappa\|k_\alpha\eta^{n+1}(t)\|_{H^3}^2 + \|\eta^{n+1}(t)\|_{H^3}^2 + \int_0^t \|\nabla u^{n+1}\|_{H^3}^2 + \|\pa_t u^{n+1}\|_{H^2}^2\,ds\\
		&\le C(\|u_0\|_{H^3}^2+\kappa\|k_\alpha\eta_0\|_{H^3}^2+\|\eta_0\|_{H^3}^2) +(C_2+C_1C_2M_2^2)\|\nabla u_0\|_{H^2}^2\\
		&\quad + C(\alpha,M_1,M_2)\int_0^t(\|u^{n+1}\|_{H^3}^2+\|\eta^{n+1}\|_{H^3}^2)\,ds.
	\end{align*}
	Then, it follows from the choice of $M_1$ that 
	\begin{align*}
		&\|u^{n+1}(t)\|_{H^3}^2+\kappa\|k_\alpha\eta^{n+1}(t)\|_{H^3}^2 + \|\eta^{n+1}(t)\|_{H^3}^2+\int^t_0\|\nabla u^{n+1}\|_{H^3}^2+\|\pa_t u^{n+1}\|_{H^2}^2\,ds\\
		&\le (C(\|u_0\|_{H^3}^2+\kappa\|k_\alpha\eta_0\|_{H^3}^2 + \|\eta_0\|_{H^3}^2)+(C_2+C_1C_2M_2^2)\|\nabla u_0\|_{H^2}^2)e^{C(\alpha,M_1,M_2)t}\\
		&<M_1,
	\end{align*}
	for all $t\in[0,T_1]$ for some $T_1>0$. We note that the time $T_1$ may depend on $\alpha$ but independent of $n$.
	Furthermore, we use Lemma \ref{lem:tilderho_H3} to derive
	\[\|\tilde{\rho}^{n+1}(t)\|_{H^3}^2\le \|\tilde{\rho}_0\|_{H^3}^2e^{C(1+M_1)t}<M_2,\]
	for all $t\in [0,T_2]$ for some $T_2>0$. Therefore, once we take $T_*:=\min\{T_1, T_2\}$, the following bound for $(\tilde{\rho}^{n+1},u^{n+1},\eta^{n+1})$ holds: for all $t\in[0,T_*]$
	\[\|u^{n+1}(t)\|_{H^3}^2+\|\eta^{n+1}(t)\|_{H^3}^2+\int^t_0\|\nabla u^{n+1}\|_{H^3}^2+\|\pa_t u^{n+1}\|_{H^2}^2\,ds\le M_1,\quad \|\tilde{\rho}^{n+1}(t)\|_{H^3}^2\le M_2.\]
	Then, the desired uniform boundedness of $(\tilde{\rho}^n,u^n,\eta^n)$ follows from the mathematical induction.
\end{proof}

\subsection{Convergence of approximate solutions}
\label{sec:3-2}
Next, we show that the sequence $(\rho^n,u^n,\eta^n)$ converges in low-order function space. To this end, we define the differences
\[\delta\rho^{n+1}:=\rho^{n+1}-\rho^n,\quad \delta u^{n+1}:= u^{n+1}-u^n,\quad \delta\eta^{n+1}:=\eta^{n+1}-\eta^n.\]
Then $(\delta \rho^{n+1}, \delta u^{n+1}, \delta \eta^{n+1})$ solves
\begin{align}
	\begin{aligned}\label{eq:diff}
		& \pa_t\delta\rho^{n+1}+\delta u^{n+1}\cdot\nabla\rho^{n+1}+u^n\cdot\nabla\delta\rho^{n+1}=0, \\
		\begin{split}
			& \rho^n\lt[\pa_t\delta u^{n+1}+(\delta u^n\cdot\nabla)u^{n+1}+(u^{n-1}\cdot\nabla)\delta u^{n+1}\rt]-\Delta\delta u^{n+1}+\frac{\delta\rho^n}{\rho^{n-1}}\Delta u^n+\nabla\delta\pi^{n+1}-\frac{\delta\rho^n}{\rho^{n-1}}\nabla\pi^n \\
			&=-\kappa\eta^n\dv k_\alpha^2\delta\eta^{n+1}-\kappa\delta\eta^n\dv k_\alpha^2\eta^n+\frac{\kappa\delta\rho^n}{\rho^{n-1}}\eta^{n-1}\dv k_\alpha^2\eta^n,
		\end{split}\\
		&\pa_t\delta\eta^{n+1}_j+\delta u^n\cdot\nabla\eta^{n+1}_j+u^{n-1}\cdot\nabla \delta\eta^{n+1}_j+\pa_j\delta u^{n+1}\cdot\eta^n+\pa_j u^n\cdot\delta\eta^n=0.
	\end{aligned}
\end{align}
We define the $L^2$-norm of the differences as
\[
X^n(t):=\|\delta\rho^n(t)\|_{L^2}^2+\|\delta u^n(t)\|_{L^2}^2+\|\delta\eta^n(t)\|_{L^2}^2.
\]

\begin{lemma}
	Under the same assumption of Theorem \ref{thm:lwp}, there exists a positive time $T_0$, which is independent of $n$, such that
	\[\lim_{n\to\infty}\sup_{t\in[0,T_0]}X^n(t)=0.\]
\end{lemma}
\begin{proof}
For simplicity, we define $\bar{C}:=1+M_1+M_2$. Then, we take an $L^2$-inner product between \eqref{eq:diff} and $(\delta\rho^{n+1},\delta u^{n+1},\delta\eta^{n+1})$ to obtain the following energy estimates:
\[
\frac{1}{2}\frac{d}{dt}\|\delta\rho^{n+1}\|_{L^2}^2\leq \bar{C}\|\delta u^{n+1}\|_{L^2}\|\delta\rho^{n+1}\|_{L^2},
\]
\[
\begin{aligned}
	\frac{1}{2}\frac{d}{dt}\|\delta\eta^{n+1}\|_{L^2}^2&= -\la\delta u^n\cdot\nabla\eta_j^{n+1},\delta\eta_j^{n+1}\ra -\la\pa_j \delta u^{n+1}\cdot\eta^n,\delta\eta^{n+1}_j\ra-\la\pa_j u^n\cdot\delta\eta^n,\delta\eta^{n+1}_j\ra\\
	&\le C\bar{C}\left(\|\delta u^n\|_{L^2}^2+\|\delta \eta^{n}\|_{L^2}^2+\|\delta \eta^{n+1}\|_{L^2}^2\right) + \frac{1}{2}\|\nabla\delta u^{n+1}\|_{L^2}^2,
\end{aligned}
\]
and 
\[
\begin{aligned}
	&\frac{1}{2}\frac{d}{dt}\lt(\int\rho^n|\delta u^{n+1}|^2\,dx+\kappa\|k_\alpha\delta\eta^{n+1}\|_{L^2}^2\rt)+\|\nabla\delta u^{n+1}\|_{L^2}^2 \\
	&\leq C\bar{C}\lt(\|\delta u^n\|_{L^2}+\|\delta\eta^n\|_{L^2}\rt)\lt(\|\delta u^{n+1}\|_{L^2} +\|\delta\eta^{n+1}\|_{L^2}\rt)\\
	&\quad +\frac{1}{2}\int\dv(\rho^n u^{n-1})|\delta u^{n+1}|^2\,dx-\la \delta\rho^n(\pa_t u^n-u^{n-1}\cdot\nabla u^n),\delta u^{n+1}\ra  \\
	&\quad -\kappa\lt\la \eta^n,\pa_jk_\alpha^2\delta\eta^{n+1}_j\delta u^{n+1}\rt\ra-\kappa\lt\la \eta^n,\pa_j\delta u^{n+1}k_\alpha^2\delta\eta^{n+1}_j\rt\ra -\kappa\lt\la u^{n-1}\cdot\nabla \delta\eta^{n+1}_j, k_\alpha^2\delta\eta^{n+1}_j\rt\ra\\
	&\leq C\bar{C}^2\lt(\|\delta u^n\|_{L^2}+\|\delta\eta^n\|_{L^2}+\|\delta\rho^n\|_{L^2}\rt)\lt(\|\delta u^{n+1}\|_{L^2} +\|\delta\eta^{n+1}\|_{L^2}\rt)\\
	&\quad+C\bar{C}^2(1+\alpha)\lt(\|\delta u^{n+1}\|_{L^2}^2+\|\delta\eta^{n+1}\|_{L^2}^2\rt)+C\|\pa_t u^n\|_{H^2}\|\delta\rho^n\|_{L^2}\|\delta u^{n+1}\|_{L^2}
\end{aligned}
\]
Therefore, using the boundedness of $\rho^n$, we obtain
\[
\frac{d}{dt}X^{n+1} \leq C\lt(\bar{C}^2+\|\pa_t u^n\|_{H^2}^2\rt) X^n+C\bar{C}^2(1+\alpha)X^{n+1}.
\]
By Gr\"onwall's lemma, we have
\[\begin{aligned}
X^{n+1}(t) &\leq \lt(X^{n+1}(0)+C\int^t_0 \lt(\bar{C}^2+\|\pa_t u^n(s)\|_{H^2}^2\rt) X^n(s)\,ds\rt)e^{C\bar{C}^2(1+\alpha)t} \\
&\leq C\lt(\bar{C}^2t+\int^t_0\|\pa_t u^n(s)\|_{H^2}^2\,ds\rt)e^{C\bar{C}^2(1+\alpha)t}\sup_{s\in[0,t]}X^n(s),
\end{aligned}\]
which implies that there exists $0<T_0\leq T_*$ such that
\[
\sup_{t\in[0,T_0]} X^{n+1}(t)\leq \frac{1}{2} \sup_{t\in[0,T_0]} X^n(t).
\]
Therefore, we obtain $\displaystyle \sup_{t\in[0,T_0]} X^n(t)\rightarrow 0$ as $n\rightarrow\infty$.
\end{proof}

To sum up, there exists a limit function $(\rho-\bar{\rho},u,\eta)$ such that
\[
\lim_{n\rightarrow\infty}(\tilde{\rho}^{n}, u^{n}, \eta^{n}) = (\rho-\bar{\rho}, u,\eta)\in C([0,T_0];L^2(\R^d)).
\]
It is standard to show that the limit solution $(\rho-\bar{\rho}, u,\eta)$ is included in $C([0,T_0];H^3(\R^d))$ from \eqref{uniform_bound} and it solves
\begin{subequations}
	\begin{align}
		&\pa_t\rho+u\cdot\nabla\rho=0, \label{limit eq a}\\
		& \rho\lt[\pa_t u+(u\cdot\nabla)u\rt]-\Delta u+\nabla \pi =-\kappa\nabla\rho  k_\alpha^2\Delta\rho,\quad \dv u=0, \\
		&\pa_t\eta_j+u\cdot\nabla\eta_j+\pa_j u\cdot\eta=0, \label{limit eq c}
	\end{align}
\end{subequations}
subject to initial data:
\[
(\rho, u,\eta)(0,x)= (\rho_0,u_0,\nabla\rho_0)(x).
\] 
On the other hand, from the uniqueness of the following equation
\[
\begin{aligned}
	&\pa_t(\eta_j-\pa_j\rho)+u\cdot\nabla(\eta_j-\pa_j\rho)+\pa_j u\cdot(\eta-\nabla\rho)=0, \\
	&(\eta-\nabla\rho)(0,x)=0,
\end{aligned}
\]
we have $\eta=\nabla\rho$, which implies that $(\rho-\bar{\rho},u)\in C([0,T_0];H^4(\R^d))\times C([0,T_0];H^3(\R^d))$ solves \eqref{NSK-relax}.

\subsection{Uniform-in-$\alpha$ estimate and proof of Theorem \ref{thm:lwp}} \label{sec:3-3}

Until now, the existence time $T_0$ may depend on $\alpha$, since the estimates for $H^3$-norms are depend on $\alpha$. In this subsection, we will show that for each fixed $\alpha>0$, the existence time $T_0$ is bounded below uniformly in $\alpha$, and consequently, $T_0$ can be chosen independently of $\alpha$, which completes the proof of Theorem \ref{thm:lwp}. To this end, we first fix $\alpha>0$ and consider a local solution $(\rho,u)$ to \eqref{NSK-relax}. We conduct a similar energy estimate as before, which is independent of $\alpha$.

The following estimate is the key of the uniform-in-$\alpha$ estimate, which is analogous to Lemma \ref{lem:u_H3}, but with the $\alpha$-independent estimate.

\begin{lemma}\label{lem:u_H3_alpha_indep}
	Let $(\rho-\bar{\rho},u)\in C([0,T_0];H^4(\R^d))\times C([0,T_0];H^3(\R^d))$ be the local solution to \eqref{NSK-relax}. Then, there exists a positive constant $C$, which is independent of $\alpha$ such that
	\begin{align*}
		&\|u(t)\|_{H^3}^2+\kappa\|k_\alpha \nabla\rho(t)\|_{H^3}^2 + \int_0^t \|\nabla u\|_{H^3}^2\,ds\\
		&\le C(\|u_0\|_{H^3}^2+\kappa\|k_\alpha\nabla\rho_0\|_{H^3}^2)+C\int_0^t (1+\|\tilde{\rho}\|_{H^3}^2+\|u\|_{H^3}^2+\|\nabla\rho\|_{H^3}^2)^2+\|\pa_t u\|_{H^2}^2\,ds.
	\end{align*}
\end{lemma}

\begin{proof}
	For the $L^2$-estimate, we multiply \eqref{NSK-relax}$_2$ by $u$ and take integration to derive 
	\begin{align*}
		\frac{1}{2}\frac{d}{dt}\int \rho|u|^2 +\kappa|k_\alpha\nabla\rho|^2\,dx +\|\nabla u\|_{L^2}^2 &= -\frac{1}{2}\int \dv(\rho u)|u|^2\,dx -\int (\rho u\cdot \nabla) u\cdot u \,dx=0.
	\end{align*}
	To control the high-order derivative, we take $D^\gamma$ with $|\gamma|\le 3$ to \eqref{NSK-relax}$_2$ and take $L^2$-inner product with $D^\gamma u$ to obtain
	
	\begin{align*}
		&\frac{1}{2}\frac{d}{dt}\left(\int\rho |D^\gamma u|^2\,dx\right)+\|\nabla D^\gamma u\|_{L^2}^2\\
		&= -\frac{1}{2}\int \dv(\rho u)|D^\gamma u|^2\,dx -\la D^\gamma(\rho\pa_t u)-\rho D^\gamma\pa_t u,D^\gamma u\ra \\
		&\quad -\la D^\gamma((\rho u\cdot\nabla)u),D^\gamma u\ra-\la D^\gamma(\kappa\nabla\rho k_\alpha^2\Delta\rho),D^\gamma u\ra\\
		&=-\la D^\gamma(\rho\pa_t u)-\rho D^\gamma\pa_t u,D^\gamma u\ra -\la D^\gamma((\rho u\cdot\nabla)u)-(\rho u\cdot\nabla)D^\gamma u,D^\gamma u\ra\\
		&\quad -\la D^\gamma(\kappa\nabla\rho k_\alpha^2\Delta\rho)-\kappa\nabla\rho D^\gamma (k^2_\alpha\Delta\rho),D^\gamma u\ra-\la\kappa\nabla\rho D^\gamma(k^2_\alpha\Delta\rho), D^\gamma u\ra=\sum_{i=1}^4 I_{3i}.
	\end{align*}
	Since the terms $I_{31}$ and $I_{32}$ are the same form as in Lemma \ref{lem:u_H3}, it is straightforward to obtain 
	\begin{align*}
		I_{31}\le C\|\tilde{\rho}\|_{H^3}\|\pa_t u\|_{H^2}\|u\|_{H^3},\quad I_{32}\le C(1+\|\tilde{\rho}\|_{H^3})\|u\|_{H^3}^3.
	\end{align*}
	Next, we use Lemma \ref{lem:est-k_alpha} to estimate $I_{33}$ as
	\[I_{33}\le C(\|\nabla\rho\|_{H^3}\|k^2_\alpha\Delta\rho\|_{L^\infty} + \|\nabla^2\rho\|_{L^\infty}\|k^2_\alpha\Delta\rho\|_{H^2})\|u\|_{H^3}\le C\|\nabla\rho\|_{H^3}^2\|u\|_{H^3}.\]
	For $I_{34}$, we use
	\[\pa_t D^\gamma\rho =-D^\gamma(u\cdot\nabla\rho)=-(D^\gamma(u\cdot\nabla\rho)-D^\gamma u\cdot\nabla\rho)-D^\gamma u \cdot\nabla\rho\]
	to obtain
	\begin{align*}
		I_{34}&=-\la\kappa\nabla\rho  D^\gamma(k^2_\alpha\Delta\rho), D^\gamma u\ra=\kappa\la D^\gamma(k^2_\alpha\Delta\rho),\pa_t D^\gamma\rho+(D^\gamma(u\cdot\nabla\rho)- D^\gamma u\cdot\nabla\rho)\ra\\
		&=-\frac{\kappa}{2}\frac{d}{dt}\|D^\gamma k_\alpha\nabla\rho\|_{L^2}^2+\kappa\la D^\gamma (k^2_\alpha \Delta\rho),D^\gamma(u\cdot\nabla\rho)-D^\gamma u\cdot\nabla\rho\ra.
	\end{align*}
	On the other hand, the second term of the right-hand side of the above equality can be estimated as
	\begin{align*}
		&\la D^\gamma(k^2_\alpha\Delta\rho),D^\gamma(u\cdot\nabla\rho)-D^\gamma u\cdot\nabla\rho\ra\\
		&=\la D^\gamma(k^2_\alpha\Delta\rho),D^\gamma(u\cdot\nabla\rho)-D^\gamma u\cdot\nabla\rho-u\cdot D^\gamma\nabla\rho\ra+\la D^\gamma(k^2_\alpha\Delta\rho),u\cdot D^\gamma\nabla\rho\ra\\
		&\le C\|u\|_{H^3}\|\nabla\rho\|^2_{H^3}+\la D^\gamma(k^2_\alpha\Delta\rho),u\cdot D^\gamma\nabla\rho\ra
	\end{align*}
	where we used the Leibniz rule and Lemma \ref{lem:sobolev}.
	Again, we further split the second term of the right-hand side as
	\begin{align*}
		\int D^\gamma(k^2_\alpha \pa_{ii}\rho)u_j D^\gamma\pa_j\rho\,dx&=\int D^\gamma(k^2_\alpha\pa_{ii}\rho)u_j D^\gamma\pa_jk_\alpha^2\rho\,dx+\int D^\gamma(k^2_\alpha\pa_{ii}\rho)u_j D^\gamma\pa_j(I-k_\alpha^2)\rho\,dx\\
		&=I_{341}+I_{342}. 
	\end{align*}
	First, using Lemma \ref{lem:est-k_alpha}, $I_{341}$ can be bounded as
	\[I_{341}=-\int D^\gamma(k^2_\alpha\pa_i\rho)\pa_iu_j D^\gamma\pa_jk^2_\alpha\rho\,dx\le C\|u\|_{H^3}\|\nabla\rho\|_{H^3}^2.\]
	Next, we use the identity
	\[I-k^2_\alpha = -\frac{1}{\alpha^2}k_\alpha^2\Delta,\]
	which can be easily obtained after taking Fourier transform, to compute $I_{342}$ as
	\[I_{342}=-\frac{1}{\alpha^2}\int D^\gamma(k^2_\alpha\Delta\rho)u_j D^\gamma\pa_j k^2_\alpha \Delta\rho\,dx = -\frac{1}{2\alpha^2}\int u_j \pa_j|D^\gamma(k^2_\alpha\Delta\rho)|^2\,dx=0.\]
	Thus, we conclude that
	\[I_{34}\le -\frac{\kappa}{2}\frac{d}{dt}\|D^\gamma k_\alpha \nabla \rho\|_{L^2}^2+C\|u\|_{H^3}\|\nabla\rho\|_{H^3}^2.\]
	Combining all the estimates of $I_{3i}$ with $i=1,2,3, 4$, we obtain
	
	\begin{align*}
		&\frac{1}{2}\frac{d}{dt}\left(\int\rho |D^\gamma u|^2\,dx+\kappa\|D^\gamma k_\alpha\nabla\rho\|_{L^2}^2\right)+\|\nabla D^\gamma u\|_{L^2}^2 \le C(1+\|\tilde{\rho}\|_{H^3}^2+\|u\|_{H^3}^2+\|\nabla\rho\|_{H^3}^2)^2+\|\pa_t u\|_{H^2}^2,
	\end{align*}
	where $C$ is now independent of $\alpha$. After integrating with respect to time and using the lower and upper boundedness of $\rho$, we get the desired estimate.
\end{proof}

The remaining estimate is merely a repetition of Lemma \ref{lem:tilderho_H3}, Lemma \ref{lem:eta_H3}, Lemma \ref{lem:patu} and Lemma \ref{lem:patu_H2}. Since there is no dependency on $\alpha$ in these preceding lemmas, we can follow the same strategy to obtain the following bounds: there exists a positive constant $C$, which is independent of $\alpha$, such that
\begin{align*}
	\|\tilde{\rho}(t)\|_{H^3}^2 &\le \|\tilde{\rho}_0\|_{H^3}^2+C\int^t_0\|u\|_{H^3}\|\tilde{\rho}\|_{H^3}^2\,ds, \\
	\|\nabla \rho(t)\|_{H^3}^2 &\le \|\nabla\rho_0\|_{H^3}^2+C\int^t_0\lt(\|u\|_{H^3}\|\nabla\rho\|_{H^3}^2+\|\nabla\rho\|_{H^3}^4\rt)\,ds+\delta_0\int^t_0\|\nabla u\|_{H^3}^2\,ds, \\
	\int_0^t \|\pa_t u\|_{L^2}^2\,ds +\|\nabla u(t)\|_{L^2}^2 &\le C\|\nabla u_0\|_{L^2}^2+C\int_0^t (\|u\|_{H^3}^2+\|\nabla\rho\|_{H^3}^2)^2\,ds,\\
	\|\nabla u(t)\|_{H^2}^2+\int_0^t \|\pa_t u\|_{H^2}^2 &\le C\|\nabla u_0\|_{H^2}^2 +C\int_0^t (1+\|\tilde{\rho}\|_{H^3}^2)(\|u\|_{H^3}^2+\|\nabla\rho\|_{H^3}^2)^2 +\|\tilde{\rho}\|_{H^3}^4\|\pa_t u\|_{L^2}^2\,ds.
\end{align*}

Now we are ready to find the existence time $T_0$ which is bounded below uniformly in $\alpha$. We first define a non-decreasing function
\[
\mathcal{M}(t):= \lt(1+\|\tilde{\rho}_0\|_{H^4}^2+\|u_0\|_{H^3}^2\rt)+\int^t_0 \lt(1+\|\tilde{\rho}\|_{H^4}^2+\|u\|_{H^3}^2\rt)^2\,ds,
\]
to obtain that 
\[
\|\tilde{\rho}(t)\|_{H^3}^2+\int_0^t \|\pa_t u\|_{L^2}^2\,ds \le C\mathcal{M}(t)
\]
and
\begin{align*}
	&\|u(t)\|_{H^3}^2+\kappa\|k_\alpha\nabla\rho\|_{H^3}^2+\|\nabla\rho\|_{H^3}^2+\int^t_0\|\nabla u\|_{H^3}^2+\|\pa_t u\|_{H^2}^2\,ds \\
	&\le C\mathcal{M}(t)+C\int^t_0 \|\tilde{\rho}\|_{H^3}^2\lt(\|u\|_{H^3}^2+\|\nabla\rho\|_{H^3}^2\rt)^2\,ds+C\int^t_0\|\tilde{\rho}\|_{H^3}^4\|\pa_t u\|_{L^2}^2\,ds \le C\mathcal{M}^3(t)
\end{align*}
using Lemma \ref{lem:u_H3_alpha_indep} and the above estimates.
Hence, we arrive at
\[
	\frac{d}{dt}\mathcal{M}(t)= \lt(1+\|\tilde{\rho}(t)\|_{H^4}^2+\|u(t)\|_{H^3}^2\rt)^2 \le C\mathcal{M}^6(t),
\]
which implies 
\[
\mathcal{M}(t)\le \frac{\mathcal{M}(0)}{\sqrt[5]{1-5C \mathcal{M}^5(0)t}}.
\]
Therefore, we choose $T_0:=\frac{1}{10C\mathcal{M}^5(0)}$ independently of $\alpha$ to deduce that for all $t\in[0,T_0]$
\[
\mathcal{M}(t)\le \sqrt[5]{2}\mathcal{M}(0)\le \sqrt[5]{2}\lt(1+\|\tilde{\rho}_0\|_{H^4}^2+\|u_0\|_{H^3}^2\rt)
\]
and hence
\[
\|\tilde{\rho}(t)\|_{H^4}^2+\|u(t)\|_{H^3}^2+\int^t_0\|\nabla u\|_{H^3}^2+\|\pa_t u\|_{H^2}^2\,ds \le C\lt(1+\|\tilde{\rho}_0\|_{H^4}^2+\|u_0\|_{H^3}^2\rt)^3.
\]
This completes the proof of Theorem \ref{thm:lwp}.

\begin{remark}
	We remark that the local existence of the local incompressible NSK system \eqref{NSK} can also be obtained by using similar iterative method. In this case, one can consider the following sequence of linear systems:
	\begin{align*}
		&\pa_t\rho^{n+1}+u^{n+1}\cdot\nabla\rho^{n+1}=0,\\
		&\rho^n[\pa_t u^{n+1}+(u^n\cdot\nabla)u^{n+1}] -\Delta u^{n+1}+\nabla \pi^{n+1}=-\kappa\eta^n \dv \eta^{n+1},\\
		&\pa_t \eta_j^{n+1}+u^n\cdot \nabla\eta_j^{n+1}+\pa_j u^{n+1}\cdot \eta^n=0, \\
		&\dv u^{n+1}=0.
	\end{align*}
	Then, by following similar estimates in Section \ref{sec:3-1} and Section \ref{sec:3-2}, one can easily get the local well-posedness of \eqref{NSK}.
\end{remark}

\section{Convergence towards the local incompressible NSK equations}\label{sec:4}
In this section, we prove Theorem \ref{thm:nonlocal_to_local} by showing that the solution to the relaxed INSK model \eqref{NSK-relax} converges to the solution to the local INSK model \eqref{NSK}, as $\alpha\to\infty$. Precisely, we fix $\kappa=1$ for simplicity, and let $(\rho^\alpha,u^\alpha)$ be the solution to \eqref{NSK-relax}
and $(\rho,u)$ be the solution to the standard INSK model
\begin{align}
\begin{aligned}\label{NSK-original}
	&\pa_t\rho+u\cdot\nabla\rho = 0,\\
	&\rho(\pa_t u + (u\cdot\nabla)u)-\Delta u +\nabla \pi = -\nabla\rho \Delta\rho,\\
	&\dv u = 0.
\end{aligned}
\end{align}

We note that the local well-posedness of \eqref{NSK-original} is already proved in \cite{YYZ15}. Specifically, there exists a solution $(\rho-\bar{\rho},u)\in C([0,T_0];H^4(\R^d))\times C([0,T_0];H^3(\R^d))$ to \eqref{NSK-original} such that
\begin{equation}\label{est:limit}
	\|\rho-\bar{\rho}\|_{H^4}^2+\|u\|_{H^3}^2+\int_0^t\|\nabla u\|_{H^3}^2+\|\pa_tu\|_{H^2}^2\,ds\le C.
\end{equation}

\subsection{Difference estimates}
To show the desired convergence in Theorem \ref{thm:nonlocal_to_local}, we define a difference $(\delta\rho,\delta u,\delta\pi):=(\rho^\alpha-\rho,u^\alpha-u,\pi^\alpha-\pi)$. Then, $(\delta\rho,\delta u,\delta\pi)$ satisfies

\begin{align}
\begin{aligned}\label{INSK-diff}
	&\pa_t\delta\rho + u^\alpha \cdot\nabla \delta\rho + \delta u \cdot\nabla\rho = 0,\\
	&\rho^\alpha [\pa_t \delta u+(u^\alpha\cdot\nabla)\delta u+ (\delta u \cdot\nabla)u ] - \Delta \delta u + \frac{\delta \rho}{\rho}\Delta u + \nabla\delta\pi -\frac{\delta\rho}{\rho}\nabla\pi \\
	&=-\nabla\delta\rho k_\alpha^2\Delta\rho^\alpha-\nabla\rho k_\alpha^2\Delta \delta\rho -\nabla\rho(k_\alpha^2-I)\Delta\rho +\frac{\delta\rho}{\rho}\nabla\rho\Delta\rho. 
\end{aligned}
\end{align}

We start with the estimate for the $H^2$-norm of the difference $\delta \rho$.

\begin{lemma}\label{lem:delta_rho_H2}
	Let $(\rho^\alpha,u^\alpha)$ be the local solution to \eqref{NSK-relax} and $(\rho,u)$ be the local solution to \eqref{NSK-original}. Then, there exists small enough positive constant $\delta_0$ such that 
	\[\frac{d}{dt}\|\delta\rho\|_{H^2}^2\le C\|\delta\rho\|_{H^2}^2+C\|\delta u\|_{H^1}^2+\delta_0\|\nabla^2\delta u\|_{L^2}^2,\]
\end{lemma}
\begin{proof}
We apply $D^\gamma$ with $|\gamma|\le 2$ to \eqref{INSK-diff}$_1$ and take inner product with $D^\gamma\delta\rho$ to obtain
\begin{align*}
	\frac{1}{2}\frac{d}{dt}\|D^\gamma\delta\rho\|_{L^2}^2 &=-\la D^\gamma(u^\alpha\cdot\nabla\delta\rho), D^\gamma\delta\rho\ra-\la D^\gamma(\delta u\cdot\nabla\rho), D^\gamma\delta\rho\ra\\
	&=-\la D^\gamma(u^\alpha\cdot\nabla\delta\rho)-u^\alpha\cdot D^\gamma\nabla\delta\rho,D^\gamma\delta\rho\ra-\la D^\gamma(\delta u\cdot\nabla\rho),D^\gamma\delta\rho\ra\\
	&\le C\|\nabla u^\alpha\|_{L^\infty}\|\nabla\delta\rho\|_{L^2}\|\delta\rho\|_{H^1}+ C(\|\nabla^2 u^\alpha\|_{L^4}\|\nabla\delta\rho\|_{L^4}+\|\nabla u^\alpha\|_{L^\infty}\|\nabla^2\delta\rho\|_{L^2})\|\delta\rho\|_{H^2}\\
	&\quad + C(\|\delta u\|_{H^2}\|\nabla\rho\|_{L^\infty}+\|\delta u\|_{L^\infty}\|\nabla\rho\|_{H^2})\|\delta\rho\|_{H^2}\\
	&\le C\|\delta\rho\|_{H^2}^2 + C\|\delta u\|_{H^1}\|\delta\rho\|_{H^2} + C\|\nabla^2\delta u\|_{L^2}\|\delta\rho\|_{H^2},
\end{align*}
where we used the uniform bound of $u^\alpha$ and $\rho$. Therefore, using Young's inequality, we have
\[\frac{d}{dt}\|\delta\rho\|_{H^2}^2\le C\|\delta\rho\|_{H^2}^2+C\|\delta u\|_{H^1}^2+\delta_0\|\nabla^2\delta u\|_{L^2}^2.\]

\end{proof}

Next, we present $H^1$-estimate for $\delta u$. 

\begin{lemma}\label{lem:delta_u_H1}
	Let $(\rho^\alpha,u^\alpha)$ be the local solution to \eqref{NSK-relax} and $(\rho,u)$ be the local solution to \eqref{NSK-original}. Then
	\begin{align*}
		&\|\delta u(t)\|_{H^1}^2+\|k_\alpha\nabla\delta\rho(t)\|_{H^1}^2+\int_0^t \|\nabla\delta u\|_{H^1}^2\,ds\\
		&\le C\int_0^t(1+\|\pa_t u\|_{H^2}^2)\left(\|\delta\rho\|_{H^2}^2+\|\delta u\|_{H^1}^2\right)+\|\pa_t\delta u\|_{L^2}^2\,ds + \frac{Ct}{\alpha^4}.
	\end{align*}
\end{lemma}

\begin{proof}
We take $D^\gamma$ with $|\gamma|\le 1$ to \eqref{INSK-diff}$_2$ and take an inner product with $D^\gamma \delta u$ to derive

\begin{align}
\begin{aligned}\label{est_du}
	&\frac{1}{2}\frac{d}{dt}\int \rho^\alpha |D^\gamma \delta u|^2\,dx+\|\nabla D^\gamma\delta u\|_{L^2}^2\\
	&= -\frac{1}{2}\int \dv(\rho^\alpha u^\alpha)|D^\gamma\delta u|^2\,dx -\la D^\gamma(\rho^\alpha \pa_t\delta u)-\rho^\alpha D^\gamma\pa_t\delta u,D^\gamma \delta u\ra\\
	&\quad  - \la D^\gamma(\rho^\alpha u^\alpha\cdot\nabla \delta u),D^\gamma\delta u\ra - \la D^\gamma(\rho^\alpha\delta u\cdot\nabla u),D^\gamma\delta u\ra \\
	&\quad -\lt\la D^\gamma\lt[\frac{\delta\rho}{\rho}(\Delta u-\nabla\pi-\nabla\rho\Delta\rho)\rt],D^\gamma\delta u\rt\ra  \\
	&\quad -\la D^\gamma(\nabla\delta\rho k^2_\alpha \Delta\rho^\alpha),D^\gamma\delta u\ra -\la D^\gamma(\nabla\rho k^2_\alpha \Delta\delta\rho),D^\gamma\delta u\ra -\la D^\gamma(\nabla\rho (k^2_\alpha-I) \Delta\rho),D^\gamma\delta u\ra.
\end{aligned}
\end{align}
On the other hand, we use
\begin{align*}
	\pa_t D^\gamma\delta\rho &= - D^\gamma (u^\alpha\cdot\nabla\delta\rho) - D^\gamma(\delta u\cdot\nabla\rho)\\
	& = -D^\gamma (u^\alpha\cdot\nabla\delta\rho) - (D^\gamma(\delta u\cdot\nabla\rho)-D^\gamma\delta u\cdot\nabla\rho)-D^\gamma\delta u\cdot\nabla\rho
\end{align*}
to obtain
\begin{align*}
	&-\la D^\gamma(\nabla\rho k^2_\alpha\Delta\delta\rho),D^\gamma\delta u\ra \\
	&= -\la D^\gamma(\nabla\rho k^2_\alpha\Delta\delta\rho)-\nabla\rho D^\gamma(k^2_\alpha\Delta\delta\rho),D^\gamma\delta u\ra - \la\nabla\rho D^\gamma(k^2_\alpha\Delta\delta\rho),D^\gamma\delta u\ra\\
	&=-\la D^\gamma(\nabla\rho k^2_\alpha\Delta\delta\rho)-\nabla\rho D^\gamma(k^2_\alpha\Delta\delta\rho),D^\gamma\delta u\ra\\
	&\quad +\la D^\gamma(k^2_\alpha\Delta\delta\rho),D^\gamma\delta\rho_t +D^\gamma(u^\alpha\cdot\nabla\delta\rho)+(D^\gamma(\delta u\cdot\nabla\rho)-D^\gamma\delta u\cdot\nabla\rho)\ra\\
	&=-\la D^\gamma(\nabla\rho k^2_\alpha\Delta\delta\rho)-\nabla\rho D^\gamma(k^2_\alpha\Delta\delta\rho),D^\gamma\delta u\ra\\
	&\quad -\frac{1}{2}\frac{d}{dt}\|D^\gamma k_\alpha\nabla\delta\rho\|_{L^2}^2 +\la D^\gamma(k^2_\alpha\Delta\delta\rho), D^\gamma(u^\alpha\cdot\nabla\delta\rho)+(D^\gamma(\delta u\cdot\nabla\rho)-D^\gamma\delta u\cdot\nabla\rho)\ra.
\end{align*}
Therefore, we further estimate \eqref{est_du} as
\begin{align*}
	&\frac{1}{2}\frac{d}{dt}\left(\int \rho^\alpha |D^\gamma \delta u|^2\,dx+\|D^\gamma k_\alpha\nabla\delta\rho\|_{L^2}^2\right)+\|\nabla D^\gamma\delta u\|_{L^2}^2\\
	&= -\frac{1}{2}\int \dv(\rho^\alpha u^\alpha)|D^\gamma\delta u|^2\,dx -\la D^\gamma(\rho^\alpha \pa_t\delta u)-\rho^\alpha D^\gamma\pa_t\delta u,D^\gamma \delta u\ra\\
	&\quad  -\la D^\gamma(\rho^\alpha u^\alpha\cdot\nabla \delta u), D^\gamma\delta u\ra - \la D^\gamma(\rho^\alpha\delta u\cdot\nabla u),D^\gamma \delta u\ra \\
	&\quad  -\lt\la D^\gamma\lt[\frac{\delta\rho}{\rho}(\Delta u-\nabla\pi-\nabla\rho\Delta\rho)\rt],D^\gamma\delta u\rt\ra  \\
	&\quad -\la D^\gamma(\nabla\delta\rho k^2_\alpha \Delta\rho^\alpha),D^\gamma\delta u\ra -\la D^\gamma(\nabla\rho (k^2_\alpha-I) \Delta\rho),D^\gamma\delta u\ra \\
	&\quad -\la D^\gamma(\nabla\rho k^2_\alpha\Delta\delta\rho)-\nabla\rho D^\gamma(k^2_\alpha\Delta\delta\rho),D^\gamma\delta u\ra\\
	&\quad +\la D^\gamma(k^2_\alpha\Delta\delta\rho), D^\gamma(u^\alpha\cdot\nabla\delta\rho)\ra+\la D^\gamma(k^2_\alpha\Delta\delta\rho),D^\gamma(\delta u\cdot\nabla\rho)-D^\gamma\delta u\cdot\nabla\rho\ra=:\sum_{i=1}^{10} J_i.
\end{align*}

In the following, we estimate the terms $J_i$ one-by-one. We use the uniform boundedness of $H^4\times H^3$-norms of $(\rho^\alpha-\bar{\rho},u^\alpha)$ and $(\rho-\bar{\rho},u)$ to obtain

\begin{align*}
	J_1+J_3 & = -\la D^\gamma(\rho^\alpha u^\alpha\cdot\nabla\delta u)-\rho^\alpha u^\alpha \cdot \nabla D^\gamma\delta u,D^\gamma\delta u\ra \le C\|\nabla(\rho^\alpha u^\alpha)\|_{L^\infty}\|\delta u\|_{H^1}^2 \leq C\|\delta u\|_{H^1}^2,\\
	J_2 &\le C\|\nabla {\rho}^\alpha\|_{L^\infty}\|\pa_t\delta u\|_{L^2}\|\delta u\|_{H^1}\le C\|\pa_t\delta u\|_{L^2}\|\delta u\|_{H^1}\le C(\|\delta u\|_{H^1}^2+\|\pa_t\delta u\|_{L^2}^2),\\	
	J_4&\le C\|\rho^\alpha\delta u\cdot\nabla u\|_{H^1}\|\delta u\|_{H^1}\le C\|\delta u\|_{H^1}^2,\\
	J_5	&= -\lt\la D^\gamma\lt[\delta\rho (\pa_t u+u\cdot\nabla u)\rt],D^\gamma\delta u\rt\ra \\
	&\le C(\|\delta\rho\|_{H^1}\|\pa_t u+u\cdot\nabla u\|_{L^\infty}+\|\delta\rho\|_{L^\infty}\|\pa_t u+u\cdot\nabla u\|_{H^1})\|\delta u\|_{H^1} \\
	&\le C(1+\|\pa_t u\|_{H^2})\|\delta\rho\|_{H^2}\|\delta u\|_{H^1}\le C(1+\|\pa_t u\|_{H^2}^2)(\|\delta\rho\|_{H^2}^2+\|\delta u\|_{H^1}^2),\\
	J_6& \le C(\|\nabla\delta\rho\|_{H^1}\|k_\alpha^2\Delta\rho^\alpha\|_{L^\infty}+\|\nabla\delta\rho\|_{L^4}\|\nabla k_\alpha^2\Delta\rho^\alpha\|_{L^4})\|\delta u\|_{H^1} \le C(\|\delta\rho\|_{H^2}^2+\|\delta u\|_{H^1}^2),\\
	J_{8}&\le C\|\nabla^2\rho\|_{L^\infty}\|k_\alpha^2\Delta\delta\rho\|_{L^2}\|\delta u\|_{H^1}\leq C(\|\delta\rho\|_{H^2}^2+\|\delta u\|_{H^1}^2).
\end{align*}
For $J_7$, we use Lemma \ref{lem:est-k_alpha} to estimate it as
\begin{align*}
	J_7 &\le C\|\nabla\rho\|_{L^\infty}\|(k_\alpha^2-I)\Delta\rho\|_{L^2}\|\delta u\|_{H^2} \\
	&\le C\|(k_\alpha^2-I)\Delta\rho\|_{L^2}^2+C\|\delta u\|_{H^1}^2+\tau\|\nabla \delta u\|_{H^1}^2 \\
	&\le \frac{C}{\alpha^4}\|\rho\|_{H^4}^2+ C\|\delta u\|_{H^1}^2 + \tau\|\nabla \delta u\|_{H^1}^2,
\end{align*}
where $\tau$ is a small positive constant that can be chosen arbitrarily small.
Next, we obtain
\begin{align*}
	J_{9}	&= \la D^\gamma (k^2_\alpha\Delta\delta\rho),D^\gamma(u^\alpha\cdot \nabla\delta\rho)-u^\alpha\cdot D^\gamma \nabla\delta\rho\ra+\la D^\gamma(k^2_\alpha \Delta\delta\rho),u^\alpha\cdot D^\gamma\nabla\delta\rho\ra \\
	&\le C\|\nabla u^\alpha \nabla\delta\rho\|_{H^1}\|k_\alpha^2\Delta\delta\rho\|_{L^2} \\
	&\quad + \la D^\gamma(k^2_\alpha\Delta\delta\rho),u^\alpha\cdot D^\gamma k^2_\alpha\nabla\delta\rho\ra + \la D^\gamma(k^2_\alpha\Delta\delta\rho),u^\alpha\cdot D^\gamma (I-k^2_\alpha)\nabla\delta\rho\ra\\
	&\le C\|\delta\rho\|_{H^2}^2 + C\|\nabla u^\alpha\|_{L^\infty}\|\delta\rho\|_{H^2}^2\le C\|\delta\rho\|_{H^2}^2,
\end{align*}
where we used $(I-k^2_\alpha)=-\frac{1}{\alpha^2}k^2_\alpha\Delta$ to get
\[\la D^\gamma k^2_\alpha \Delta\delta\rho,u^\alpha\cdot D^\gamma(I-k^2_\alpha)\nabla\delta\rho\ra = -\frac{1}{\alpha^2}\la D^\gamma k^2_\alpha \Delta\delta\rho, u^\alpha \cdot\nabla D^\gamma k^2_\alpha \Delta\delta\rho\ra=0.\]
Finally, we estimate $J_{10}$ as
\begin{align*}
	J_{10}&\le C\|\delta\rho\|_{H^2}\|\delta u \nabla^2 \rho\|_{H^1} \le C(\|\delta\rho\|_{H^2}^2+\|\delta u\|_{H^1}^2).
\end{align*}
Therefore, combining all the estimates of $J_{i}$ for $i=1,2,\ldots, 10$, and using the lower and upper bounds of $\rho^\alpha$, we obtain 
\begin{align*}
	&\|\delta u(t)\|_{H^1}^2+\|k_\alpha\nabla\delta\rho(t)\|_{H^1}^2+\int_0^t \|\nabla\delta u\|_{H^1}^2\,ds\\
	&\le C\int_0^t(1+\|\pa_t u\|_{H^2}^2)\left(\|\delta\rho\|_{H^2}^2+\|\delta u\|_{H^1}^2\right)+\|\pa_t\delta u\|_{L^2}^2\,ds + \frac{Ct}{\alpha^4},
\end{align*}
which is the desired estimate.
\end{proof}

To close the $H^1$-norm estimate of $\delta u$, we need to control the term
\[\int_0^t \|\pa_t\delta u\|_{L^2}^2\,ds.\]

\begin{lemma}\label{lem:pat_deltau_L2}
	Let $(\rho^\alpha,u^\alpha)$ be the local solution to \eqref{NSK-relax} and $(\rho,u)$ be the local solution to \eqref{NSK-original}. Then,
	\[	\|\nabla\delta u(t)\|_{L^2}^2+\int^t_0\|\pa_t\delta u\|_{L^2}^2\,ds \le C\int^t_0 (1+\|\pa_t u\|_{H^2}^2)(\|\delta \rho\|^2_{H^2}+\|\delta u\|_{H^1}^2)\,ds+\frac{Ct}{\alpha^4}.\]
\end{lemma}

\begin{proof}
We multiply \eqref{INSK-diff}$_2$ by $\pa_t\delta u$ to obtain
\begin{align*}
	&\int \rho^\alpha|\pa_t\delta u|^2\,dx + \frac{1}{2}\frac{d}{dt}\|\nabla\delta u\|_{L^2}^2\\
	&=-\la\rho^\alpha u^\alpha \cdot\nabla\delta u,\pa_t\delta u\ra-\la \rho^\alpha \delta u\cdot\nabla u,\pa_t\delta u\ra-\lt\la\frac{\delta \rho}{\rho}\lt(\Delta u-\nabla\pi-\nabla\rho\Delta\rho\rt),\pa_t\delta u\rt\ra\\
	&\quad -\la\nabla\delta\rho k^2_\alpha \Delta \rho^\alpha ,\pa_t\delta u\ra-\la\nabla\rho k^2_\alpha \Delta \delta\rho,\pa_t\delta u\ra -\la \nabla\rho(k^2_\alpha-I)\Delta\rho,\pa_t\delta u\ra\\
	&\le \frac{\rho_m}{2}\|\pa_t\delta u\|_{L^2}^2+C(1+\|\pa_tu\|_{L^2}^2)(\|\delta \rho\|_{H^2}^2+\|\delta u\|_{H^1}^2) +\frac{C}{\alpha^4},
\end{align*}
where we used Lemma \ref{lem:est-k_alpha} in the last inequality. Using the lower and upper bounds of $\rho^\alpha$, we derive

\[
	\|\pa_t\delta u\|_{L^2}^2+\frac{d}{dt}\|\nabla\delta u\|_{L^2}^2\le C(1+\|\pa_tu\|_{L^2}^2)(\|\delta \rho\|_{H^2}^2+\|\delta u\|_{H^1}^2) +\frac{C}{\alpha^4}.
\]
After integrating with respect to time, we get the desired estimate.
\end{proof}

\subsection{Proof of Theorem \ref{thm:nonlocal_to_local}}
We are now in a position to prove Theorem \ref{thm:nonlocal_to_local}, the nonlocal-to-local convergence. Combining Lemma \ref{lem:delta_rho_H2}, Lemma \ref{lem:delta_u_H1} and Lemma \ref{lem:pat_deltau_L2}, we deduce the following estimate on the difference $(\delta\rho,\delta u)$:
\begin{align*}
	&\|\delta\rho\|_{H^2}^2+\|\delta u\|_{H^1}^2+\|k_\alpha\nabla\delta\rho\|_{H^1}^2+\int_0^t \|\nabla\delta u\|_{H^1}^2\,ds\\
	&\le C\int_0^t(1+\|\pa_t u\|_{H^2}^2)\left(\|\delta\rho\|_{H^2}^2+\|\delta u\|_{H^1}^2\right)\,ds + \frac{Ct}{\alpha^4}.
\end{align*}
Then, Gr\"onwall's inequality yields
\[\sup_{0\le t\le T_0}\left(\|\delta\rho\|_{H^2}^2+\|\delta u\|_{H^1}^2\right)\le \frac{CT_0}{\alpha^4}\exp\left(\int_0^{T_0}(1+\|\pa_tu\|_{H^2}^2)\,ds\right)\le \frac{C(T_0)}{\alpha^4},\]
where we used \eqref{est:limit} in the last inequality. Using the boundedness of $(\delta\rho,\delta u)$ in $H^4(\R^d)\times H^3(\R^d)$ and Lemma \ref{lem:sobolev}, this completes the proof of Theorem \ref{thm:nonlocal_to_local}.

\begin{remark}
	Besides the convergence of $\rho^\alpha$ to $\rho$, we can obtain the convergence of $c^\alpha=k^2_\alpha\rho^\alpha$, which is often called as an ``order parameter" \cite{C16}. Specifically, using Lemma \ref{lem:est-k_alpha}, it holds for any $l\in[0,2]$ that 
	\begin{align*}
		\sup_{0\le t\le T_0}\|c^\alpha(t)-\rho^\alpha(t)\|_{H^{2+l}}^2=\sup_{0\le t\le T_0}\|(k^2_\alpha-I)\rho^\alpha(t)\|_{H^{2+l}}^2\le \frac{1}{\alpha^{4-2l}}\sup_{0\le t\le T_0}\|\rho^\alpha(t)\|_{H^4}^2\le \frac{C(T_0)}{\alpha^{4-2l}},
	\end{align*}
	and consequently,
	\begin{align*}
		\sup_{0\le t\le T_0}\|c^\alpha(t)-\rho(t)\|_{H^{2+l}}^2&\le 2\sup_{0\le t\le T_0}\left(\|c^\alpha(t)-\rho^\alpha(t)\|^2_{H^{2+l}}+\|\rho^\alpha(t)-\rho(t)\|_{H^{2+l}}^2\right) \le \frac{C_l(T_0)}{\alpha^{4-2l}}.
	\end{align*}
\end{remark}


\section{Convergence towards the incompressible Navier--Stokes equations}\label{sec:5}
In this section, we consider the vanishing capillarity limit and prove Theorem \ref{thm:vanishing_capillarity}. Let us fix $\alpha$ and $(\rho^\kappa,u^\kappa,\pi^\kappa)$ be the solution to \eqref{NSK-relax} with capillarity $\kappa$. Moreover, let $(\rho,u,\pi)$ be the solution to the inhomogeneous incompressible Navier-Stokes equations \eqref{NS}. Using a similar iterative method, one can easily find that the incompressible Navier-Stokes equations \eqref{NS} admits a unique solution $(\rho-\bar{\rho},u)\in C([0,T_0];H^3(\R^d))\times C([0,T_0];H^3(\R^d))$ such that
\begin{equation}\label{est:limit-capillarity}
	\|u\|_{H^3}^2+\|\tilde{\rho}\|_{H^3}^2+\int_0^t\|\nabla u\|_{H^3}^2+\|\pa_tu\|_{H^2}^2\,ds\le C.
\end{equation}

To prove the vanishing capillarity limit, we again take differences $\delta\rho :=\rho^\kappa-\rho$, $\delta u = u^\kappa-u$, and $\delta\pi = \pi^\kappa-\pi$. Then, we subtract \eqref{NSK-relax} by \eqref{NS} to obtain the equations for $(\delta\rho,\delta u,\delta \pi)$ as

\begin{align}
\begin{aligned}\label{eq:delta_rho_u}
	&\pa_t\delta\rho +\delta u\cdot\nabla\rho^\kappa+u\cdot\nabla\delta\rho=0,\\
	&\rho^\kappa[\pa_t \delta u + (u^\kappa\cdot\nabla)\delta u +(\delta u \cdot\nabla)u] -\Delta \delta u +\frac{\delta\rho}{\rho}\Delta u + \nabla\delta \pi -\frac{\delta\rho}{\rho}\nabla\pi=-\kappa\nabla\rho^\kappa k^2_\alpha\Delta\rho^\kappa.
\end{aligned}
\end{align}

We first obtain the $H^3$-estimate for $\delta\rho$.
\begin{lemma}\label{lem:delta_rho_H3_capillarity}
	Let $(\rho^\kappa,u^\kappa)$ be the local solution to \eqref{NSK-relax} and $(\rho,u)$ be the local solution to \eqref{NS}. Then, it holds that
\begin{equation}\label{delta_rho_H3_capillarity}
	\frac{d}{dt}\|\delta\rho\|_{H^3}^2\le C(\|\delta\rho\|_{H^3}^2+\|\delta u\|_{H^3}^2).
\end{equation}
\end{lemma}

\begin{proof}
We apply $D^\gamma$ with $|\gamma|\le 3$ to \eqref{eq:delta_rho_u}$_1$ and take inner product with $D^\gamma\delta\rho$ to obtain
	\begin{align*}
	\frac{1}{2}\frac{d}{dt}\|D^\gamma \delta\rho\|_{H^3}^2 	&=-\la D^\gamma(\delta u\cdot\nabla\rho^\kappa),D^\gamma\delta\rho\ra-\la D^\gamma(u\cdot\nabla\delta\rho)-u\cdot\nabla D^\gamma\delta\rho,D^\gamma \delta\rho\ra \\
	&\le C\|\delta u\|_{H^3}\|\delta\rho\|_{H^3}+C\|\delta\rho\|_{H^3}^2\le C(\|\delta\rho\|_{H^3}^2+\|\delta u\|_{H^3}^2),
	\end{align*}
	where we used the uniform bound of $\rho^\kappa$ and $u$. Summing over all $|\gamma|\le 3$, we obtain the desired estimate.
\end{proof}

Next, we obtain the $H^3$-estimate for $\delta u$. 
\begin{lemma}\label{lem:delta_u_H3_capillarity}
	Let $(\rho^\kappa,u^\kappa)$ be the local solution to \eqref{NSK-relax} and $(\rho,u)$ be the local solution to \eqref{NS}. Then, it holds that
	\begin{align*}
		\|\delta u(t)\|_{H^3}^2+\int_0^t \|\nabla\delta u\|_{H^3}^2\,ds\le C\int_0^t(1+\|\pa_t u\|_{H^2}^2)\left(\|\delta\rho\|_{H^3}^2+\|\delta u\|_{H^3}^2\right)+\|\pa_t\delta u\|_{H^2}^2\,ds + C\kappa^2t.
	\end{align*}
\end{lemma}
\begin{proof}
	For any multiindex $\gamma$ with $|\gamma|\le 3$, we take $D^\gamma$ to \eqref{eq:delta_rho_u}$_2$ and take $L^2$-inner product with $D^\gamma\delta u$ to obtain
	\begin{align*}
		&\frac{1}{2}\frac{d}{dt}\int \rho^\kappa |D^\gamma \delta u|^2\,dx+\|\nabla D^\gamma\delta u\|_{L^2}^2\\
		&= -\frac{1}{2}\int \dv(\rho^\kappa u^\kappa)|D^\gamma\delta u|^2\,dx -\la D^\gamma(\rho^\kappa \pa_t\delta u)-\rho^\kappa D^\gamma\pa_t\delta u,D^\gamma \delta u\ra\\
		&\quad - \la D^\gamma(\rho^\kappa u^\kappa\cdot\nabla \delta u),D^\gamma\delta u\ra - \la D^\gamma(\rho^\kappa\delta u\cdot\nabla u),D^\gamma\delta u\ra \\
		&\quad -\lt\la D^\gamma\lt(\frac{\delta\rho}{\rho}\Delta u\rt),D^\gamma\delta u\rt\ra  + \lt\la D^\gamma\lt(\frac{\delta\rho}{\rho}\nabla\pi\rt),D^\gamma\delta u\rt\ra\\
		&\quad -\kappa\la D^\gamma(\nabla\rho^\kappa k^2_\alpha\Delta\rho^\kappa),D^\gamma \delta u\ra=:\sum_{i=1}^7 J_i.
	\end{align*}
We estimate the terms $J_i$ one-by-one. We use the uniform boundedness of $(\rho^\kappa,u^\kappa)$ and $(\rho,u)$ to obtain
	\begin{align*}
		J_1+J_3 &= -\la D^\gamma(\rho^\kappa u^\kappa\cdot\nabla\delta u)-\rho^\kappa u^\kappa\cdot\nabla D^\gamma \delta u,D^\gamma\delta u\ra \\
		&\le C(\|\rho^\kappa u^\kappa\|_{H^3}\|\nabla\delta u\|_{L^\infty}+\|\nabla(\rho^\kappa u^\kappa)\|_{L^\infty}\|\nabla\delta u\|_{H^2})\|\delta u\|_{H^3} \le C\|\delta u\|_{H^3}^2,\\
		J_2 &\le C(\|\rho^\kappa\|_{H^3}\|\pa_t\delta u\|_{L^\infty}+\|\nabla\rho^\kappa\|_{L^\infty}\|\pa_t\delta u\|_{H^2})\|\delta u\|_{H^3} \\
		&\le C(\|\delta u\|_{H^3}^2+\|\pa_t\delta u\|_{H^2}^2),\\
		J_4&\le C\|\rho^\kappa\delta u\cdot\nabla u\|_{H^2}(\|\delta u\|_{H^3}+\|\nabla \delta u\|_{H^3}) \\
		& \le \tau\|\nabla\delta u\|_{H^3}^2+  C\|\delta u\|_{H^3}^2,\\
		J_5+J_6 &=-\la D^\gamma(\delta\rho(\pa_t u+u\cdot\nabla u)),D^\gamma \delta u\ra \\
		&\le C\|\delta\rho(\pa_t u+u\cdot\nabla u)\|_{H^2}(\|\delta u\|_{H^3}+\|\nabla\delta u\|_{H^3}) \\
		&\le \tau\|\nabla \delta u\|_{H^3}^2+C(1+\|\pa_t u\|_{H^2}^2)(\|\delta\rho\|_{H^3}^2+\|\delta u\|_{H^3}^2).
	\end{align*}
	Finally, we estimate the term $J_7$ as
	\begin{align*}
		J_7\le \kappa \|\nabla\rho^\kappa k_\alpha^2\Delta\rho^\kappa\|_{H^2}(\|\delta u\|_{H^3}+\|\nabla\delta u\|_{H^3}) \le \tau\|\nabla\delta u\|_{H^3}^2+C\|\delta u\|_{H^3}^2+C\kappa^2,
	\end{align*}
	where we used Lemma \ref{lem:est-k_alpha} and the boundedness of $\rho^\kappa$. Therefore, combining all the estimates and using the lower and upper bound of $\rho^\alpha$, we obtain 
	\begin{align*}
		\|\delta u\|_{H^3}^2+\int_0^t \|\nabla\delta u\|_{H^3}^2\,ds\le C\int_0^t(1+\|\pa_t u\|_{H^2}^2)\left(\|\delta\rho\|_{H^3}^2+\|\delta u\|_{H^3}^2\right)+\|\pa_t\delta u\|_{H^2}^2\,ds + C\kappa^2t.
	\end{align*}
\end{proof}


\begin{lemma}\label{lem:pat_deltau_H2_capillarity}
	Let $(\rho^\kappa,u^\kappa)$ be the local solution to \eqref{NSK-relax} and $(\rho,u)$ be the local solution to \eqref{NS}. Then,
	\[ \|\nabla\delta u(t)\|_{H^2}^2+\int^t_0\|\pa_t\delta u\|_{H^2}^2\,ds \le C\int^t_0(1+\|\pa_t u\|_{H^2}^2)(\|\delta \rho\|^2_{H^3}+\|\delta u\|_{H^3}^2)\,ds+C\kappa^2t.\]
\end{lemma}
\begin{proof}
	The proof follows the similar argument as in Lemma \ref{lem:pat_deltau_L2}. For the $L^2$-estimate, we have
	\begin{align*}
		&\int \rho^\kappa|\pa_t\delta u|^2\,dx + \frac{1}{2}\frac{d}{dt}\|\nabla\delta u\|_{L^2}^2\\
		&=-\la\rho^\kappa u^\kappa \cdot\nabla\delta u,\pa_t\delta u\ra-\la \rho^\kappa \delta u\cdot\nabla u,\pa_t\delta u\ra-\lt\la\frac{\delta \rho}{\rho}\Delta u,\pa_t\delta u\rt\ra+\lt\la\frac{\delta\rho}{\rho}\nabla\pi,\pa_t\delta u\rt\ra\\
		&\quad -\kappa\la \nabla\rho^\kappa k^2_\alpha\Delta\rho^\kappa, \pa_t\delta u\ra\\
		&\le \frac{\rho_m}{2}\|\pa_t\delta u\|_{L^2}^2+C(1+\|\pa_tu\|_{H^2}^2)(\|\delta u\|_{H^3}^2+\|\delta\rho\|_{H^3}^2) +C\kappa^2.
	\end{align*}
	where we used Lemma \ref{lem:est-k_alpha} and the uniform boundedness of $\rho^\kappa$ in the last inequality. Using the lower and upper bounds of $\rho^\kappa$, we have
	
	\begin{equation}\label{est:pat_deltau-L2-capillarity}
		\|\pa_t\delta u\|_{L^2}^2+\frac{d}{dt}\|\nabla\delta u\|_{L^2}^2\le C(1+\|\pa_tu\|_{H^2}^2)(\|\delta u\|_{H^3}^2+\|\delta\rho\|_{H^3}^2) +C\kappa^2.
	\end{equation}
	To derive $H^2$-estimate, we take $D^\gamma$ for $|\gamma|\le 2$ to \eqref{eq:delta_rho_u}$_2$ and take $L^2$-inner product with $D^\gamma\pa_t\delta u$ to obtain
	\begin{align*}
		&\int \rho^\kappa |\pa_t D^\gamma \delta u|^2\,dx +\frac{1}{2}\frac{d}{dt}\|D^\gamma \nabla \delta u\|_{L^2}^2\\
		& = -\la D^\gamma(\rho^\kappa u^\kappa \cdot\nabla \delta u),D^\gamma\pa_t\delta u\ra-\la D^\gamma(\rho^\kappa\pa_t\delta u)-\rho^\kappa D^\gamma \pa_t\delta u,D^\gamma\pa_t \delta u\ra\\
		&\quad -\la D^\gamma(\rho^\kappa \delta u\cdot\nabla u),D^\gamma\pa_t\delta u\ra -\lt\la D^\gamma\lt(\frac{\delta \rho}{\rho}\Delta u\rt),D^\gamma\pa_t\delta u\rt\ra+\lt\la D^\gamma\lt(\frac{\delta \rho}{\rho}\nabla\pi\rt),D^\gamma\pa_t\delta u\rt\ra\\
		&\quad -\la \kappa D^\gamma(\nabla\rho^\kappa k^2_\alpha\Delta \rho^\kappa),D^\gamma\pa_t\delta u\ra\\
		&=:\sum_{i=1}^6 K_i
	\end{align*}
	We obtain
	\begin{align*}
		K_1 &\le \tau\|\pa_t\delta u\|_{H^2}^2 + C\|\delta u\|_{H^3}^2,\\
		K_2&\le C\|\nabla \rho^\kappa\|_{L^\infty}\|\pa_t\delta u\|_{L^2}\|\pa_t\delta u\|_{H^1} + C\lt(\|\nabla^2\rho^\kappa\|_{L^\infty}\|\pa_t\delta u\|_{L^2}+\|\nabla\rho^\kappa\|_{L^\infty}\|\nabla\delta u\|_{L^2}\rt)\|\pa_t\delta u\|_{H^2} \\
		&\le C\|\pa_t\delta u\|_{L^2}^{\frac{1}{2}}\|\pa_t \delta u\|_{H^2}^{\frac{3}{2}}\le \tau \|\pa_t\delta u\|_{H^2}^2 + C\|\pa_t\delta u\|_{L^2}^2,\\
		K_3& \le \tau\|\pa_t\delta u\|_{H^2}^2 +C\|\delta u\|_{H^3}^2,\\
		K_4+	K_5 &= -\lt\la D^\gamma\lt(\delta\rho(\pa_t u+u\cdot\nabla u)\rt),D^\gamma\pa_t\delta u\rt\ra \\
		&\le \tau\|\pa_t\delta u\|_{H^2}^2+ C(1+\|\pa_tu\|_{H^2}^2)\|\delta\rho\|_{H^3}^2,
	\end{align*}
	where $\tau$ is small enough constant. Finally, $K_6$ can be directly estimated as
	\[K_6 \le \tau\|\pa_t\delta u\|_{H^2}^2 + C\kappa^2,\]
	thanks to the uniform boundedness of $\rho^\kappa$. Thus, we obtain
	\begin{equation}\label{est:pat_deltau-H2-capillarity}
		\frac{d}{dt}\|\nabla\delta u\|_{H^2}^2+\|\pa_t\delta u\|_{H^2}^2\le C(1+\|\pa_t u\|_{H^2}^2)(\|\delta\rho\|_{H^3}^2+\|\delta u\|_{H^3}^2)+ C\|\pa_t\delta u\|_{L^2}^2+C\kappa^2.
	\end{equation}
	We collect \eqref{est:pat_deltau-L2-capillarity} and \eqref{est:pat_deltau-H2-capillarity} to derive the desired estimate.
\end{proof}

Combining Lemma \ref{lem:delta_rho_H3_capillarity}, Lemma \ref{lem:delta_u_H3_capillarity} and Lemma \ref{lem:pat_deltau_H2_capillarity}, we obtain the following inequality:
\begin{align*}
	&\|\delta\rho\|_{H^3}^2+\|\delta u\|_{H^3}^2+\int_0^t \|\nabla\delta u\|_{H^3}^2+\|\pa_t\delta u\|_{H^2}^2\,ds\\
	&\le C\int_0^t(1+\|\pa_t u\|_{H^2}^2)\left(\|\delta\rho\|_{H^3}^2+\|\delta u\|_{H^3}^2\right)\,ds + C\kappa^2t.
\end{align*}
Then, we use Gr\"onwall's inequality and \eqref{est:limit-capillarity} to get
\[\sup_{0\le t\le T_0}\left(\|\delta\rho\|_{H^3}^2+\|\delta u\|_{H^3}^2\right)\le C\kappa^2T_0\exp\left(\int_0^{T_0}(1+\|\pa_tu\|_{H^2}^2)\,ds\right)\le C(T_0)\kappa^2,\]
which completes the proof of Theorem \ref{thm:vanishing_capillarity}.


\begin{thebibliography}{99}

\bibitem{AS22}
\newblock P. Antonelli and S. Spirito,
\newblock Global existence of weak solutions to the Navier--Stokes--Korteweg equations,
\newblock \textit{Ann. Inst. H. Poincar\'e C Anal. Non Lin\'eaire} \textbf{39} (2022), 171--200.

\bibitem{BC17}
\newblock C. Burtea and F. Charve,
\newblock Lagrangian methods for a general inhomogeneous incompressible Navier--Stokes--Korteweg system with variable capillarity and viscosity coefficients,
\newblock \textit{SIAM J. Math. Anal.} \textbf{49}, No. 5 (2017), 3476--3495.

\bibitem{C16}
\newblock F. Charve,
\newblock Convergence of a low order non-local Navier–Stokes–Korteweg system: the order-parameter model,
\newblock \textit{Asymptot. Anal.} \textbf{100} (2016), 153--191.

\bibitem{CH11}
\newblock F. Charve and B. Haspot,
\newblock Convergence of capillary fluid models: from the non-local to the local Korteweg model,
\newblock \textit{Indiana Univ. Math. J.} \textbf{60}, No. 6 (2011), 2021--2059.

\bibitem{CH13}
\newblock F. Charve and B. Haspot,
\newblock Existence of a global strong solution and vanishing capillarity-viscosity limit in one dimension for the Korteweg system,
\newblock \textit{SIAM J. Math. Anal.} \textbf{45}, No. 2 (2013), 469--494.

\bibitem{DS85}
\newblock J. E. Dunn and J. Serrin, 
\newblock On the thermomechanics of interstitial working,
\newblock \textit{Arch. Rational Mech. Anal.} \textbf{88}, No. 2 (1985), 95--133.

\bibitem{DD01}
\newblock R. Danchin and B. Desjardins,
\newblock Existence of solutions for compressible fluid models of Korteweg type,
\newblock \textit{Ann. Inst. H. Poincar\'e C Anal. Non Lin\'eaire} \textbf{18} (2001), 97--133.

\bibitem{H11}
\newblock B. Haspot,
\newblock Existence of global weak solution for compressible fluid moedls of Korteweg type,
\newblock \textit{J. Math. Fluid. Mech.} \textbf{13} (2011), 223--249.

\bibitem{HKMR20}
\newblock T. Hitz, J. Keim, C.-D. Munz, and C. Rohde,
\newblock A parabolic relaxation model for the Navier--Stokes--Korteweg equations,
\newblock \textit{J. Comput. Phys.} \textbf{421} (2020), 109714.

\bibitem{HKKL25}
\newblock S. Han, M.-J. Kang, J. Kim, and H. Lee,
\newblock Long-time behavior towards viscous-dispersive shock for Navier-Stokes equations of Korteweg type,
\newblock \textit{J. Differential Equations} \textbf{426} (2025), 317--387.

\bibitem{HL94}
\newblock H. Hattori and D. Li,
\newblock Solutions for two dimensional system for materials of Korteweg type,
\newblock \textit{SIAM J. Math. Anal.} \textbf{25} (1994), 85--98.

\bibitem{JB23}
\newblock A. Jabour and A. Bouidi,
\newblock Local existence and uniqueness of strong solutions to the density-dependent incompressible Navier--Stokes--Korteweg system,
\newblock \textit{J. Math. Anal. Appl.} \textbf{517} (2023), 126611.

\bibitem{K01}
\newblock D. J. Korteweg,
\newblock Sur la forme que pennent les \'equations du mouvements des fluides si l'on tient compte des forces capillaires caus\'ees par des variations de densit\'e , consid\'erables mais connues et sur la th\'eorie de la capillarit\'e dans l'hypoth\'ese d'une variation continue de la densit\'e,
\newblock \textit{Arch. N\'eerl. Sci. Exactes Nat.} \textbf{6} (1901), 1--24.

\bibitem{L20}
\newblock H. Li,
\newblock A blow-up criterion for the density-dependent Navier--Stokes--Korteweg equations in dimension two,
\newblock \textit{Acta Appl. Math.} \textbf{166} (2020), 73--83.

\bibitem{L21}
\newblock H. Li,
\newblock A blow-up criterion for the strong solutions to the nonhomogeneous Navier--Stokes--Korteweg equations in dimension three,
\newblock \textit{Appl. Math.} \textbf{66} (2021), 43--55.

\bibitem{LZZ24}
\newblock F. Li, S. Zhang, and Z. Zhang,
\newblock Uniform regularity and zero capillarity-viscosity limit for an inhomogeneous incompressible fluid model of Korteweg type in half-space,
\newblock \textit{Nonlinearity} \textbf{37} (2024), 035002.

\bibitem{R05}
\newblock C. Rohde,
\newblock On local and non-local Navier--Stokes--Korteweg systems for liquid-vapor phase transitions,
\newblock \textit{Z. Angew. Math. Mech.} \textbf{85}, No. 12 (2005), 839--857.

\bibitem{SBGLR06}
\newblock M. Sy, D. Bresch, F. Guill\'en-Gonz\'alez, J. Lemoine, and M. A. Rodr\'iguez-Bellido,
\newblock Local strong solution for the incompressible Korteweg model,
\newblock \textit{C. R. Acad. Sci. Paris, Ser. I} \textbf{342} (2006), 169--174.

\bibitem{TZ14}
\newblock Z. Tan and R. Zhang,
\newblock Optimal decay rates of the compressible fluid models of Korteweg type,
\newblock \textit{Z. Angew. Math. Phys.} \textbf{65} (2014), 279--300.

\bibitem{W94}
\newblock J. D. van der Waals,
\newblock Thermodynamische theorie der kapillarit\"at unter voraussetzung stetiger dichte\"anderung,
\newblock \textit{Z. Phys. Chem.} \textbf{13} (1894), 657--725.

\bibitem{W17}
\newblock T. Wang,
\newblock Unique solvability for the density-dependent incompressible Navier--Stokes--Korteweg system,
\newblock \textit{J. Math. Anal. Appl.} \textbf{455} (2017), 606--618.

\bibitem{Wpre}
\newblock S. Wang,
\newblock Global existence and uniqueness of the density-dependent incompressible Navier--Stokes--Korteweg system with variable capillarity and viscosity coefficients,
\newblock \textit{arXiv preprint} arXiv:2408.11675.

\bibitem{WZ24}
\newblock P. Wang and Z. Zhang,
\newblock Vanishing capilarity-viscosity limit of the incompressible Navier--Stokes--Korteweg equations with slip boundary condition,
\newblock \textit{Nonlinear Analysis} \textbf{243} (2024), 113526.

\bibitem{YYZ15}
\newblock J. Yang, L. Yao, and C. Zhu,
\newblock Vanishing capillarity-viscosity limit for the incompressible inhomogeneous fluid models of Korteweg type,
\newblock \textit{Z. Angew. Math. Phys.} \textbf{66}, No. 5 (2015), 2285--2303.

\bibitem{ZY10}
\newblock T. Zhong and W. Yanjin,
\newblock Strong solutions for the incompressible fluid models of Koreteweg type,
\newblock \textit{Acta Math. Sci.} \textbf{30B} (2010), 799--809.



\end{thebibliography}
\end{document}